\documentclass[10pt]{article}
\usepackage{graphicx}
\usepackage{amssymb}
\usepackage{enumerate}
\usepackage{amsmath,amsthm,amsfonts,amssymb,verbatim,graphicx,subfigure}
\usepackage{todonotes}
\usepackage{fullpage}
\usepackage{algorithm}
\usepackage{algpseudocode}
\usepackage{epsfig,multirow}
\usepackage[graphicx]{realboxes}

\newtheorem{theorem}{Theorem}
\newtheorem{definition}[theorem]{Definition}
\newtheorem{remark}[theorem]{Remark}
\newtheorem{lemma}[theorem]{Lemma}

\newtheorem{proposition}[theorem]{Proposition}
\newtheorem{corollary}[theorem]{Corollary}

\newcommand{\inp}{\smallskip\makebox[1.2cm][l]{\em Input}}
\newcommand{\out}{\makebox[1.2cm][l]{\em Output}}

\newcommand{\MWC}{\textsc{Maximum $G$-Coverage}}
\newcommand{\MS}{\textsc{group model projection}}
\newcommand{\cS}{\mathcal S}
\newcommand{\cG}{\mathcal G}
\newcommand{\fG}{\mathfrak G}

\def\d{\delta}

\def\e{\epsilon}

\def\Gm{\Gamma}
\def\A{\mathbf{A}}
\def\RR{\mathbb{R}}
\newcommand{\edges}{\mathcal E}
\newcommand{\obs}{\mathbf y}
\newcommand{\x}{\mathbf x}
\newcommand{\bu}{\mathbf u}
\newcommand{\bv}{\mathbf v}
\newcommand{\w}{\mathbf w}
\newcommand{\z}{\mathbf z}
\newcommand{\sol}{\widehat{\x}}
\newcommand{\noise}{\mathbf e}
\newcommand{\NN}{[N]}
\newcommand{\bigO}{\mathcal O}
\newcommand{\support}{\mathcal S}
\newcommand{\0}{\boldsymbol{0}}
\newcommand{\med}{\mathcal M}
\newcommand{\supp}{\text{supp}}

\newcommand{\NP}{\mbox{\emph{NP}}}

\begin{document}
\title{Discrete Optimization Methods for Group Model Selection in Compressed Sensing}

\author{Bubacarr Bah\thanks{AIMS South Africa, \& Stellenbosch University, 
              6 Melrose Road, Muizenberg, Cape Town 7945, South Africa} \and Jannis Kurtz\thanks{Chair for Mathematics of Information Processing, RWTH Aachen University, Pontdriesch 10, 52062 Aachen, Germany} \and Oliver Schaudt\thanks{Chair for Mathematics of Information Processing, RWTH Aachen University, Pontdriesch 10, 52062 Aachen, Germany}}

\date{} 

\providecommand{\keywords}[1]{\textit{#1}}


\maketitle

\begin{abstract}
In this article we study the problem of signal recovery for group models. More precisely for a given set of groups, each containing a small subset of indices, and for given linear sketches of the true signal vector which is known to be group-sparse in the sense that its support is contained in the union of a small number of these groups, we study algorithms which successfully recover the true signal just by the knowledge of its linear sketches. We derive model projection complexity results and algorithms for more general group models than the state-of-the-art. We consider two versions of the classical Iterative Hard Thresholding algorithm (IHT). The classical version iteratively calculates the exact projection of a vector onto the group model, while the approximate version (AM-IHT) uses a head- and a tail-approximation iteratively. We apply both variants to group models and analyse the two cases where the sensing matrix is a Gaussian matrix and a model expander matrix.

To solve the exact projection problem on the group model, which is known to be equivalent to the maximum weight coverage problem, we use discrete optimization methods based on dynamic programming and Benders' Decomposition. The head- and tail-approximations are derived by a classical greedy-method and LP-rounding, respectively.

\end{abstract}

\section{Introduction}

In many applications involving sensors or sensing systems an unknown sparse signal has to be recovered from a relatively small number of measurements. The reconstruction problem in standard compressed sensing attempts to recover an unknown $k$-sparse signal $\x \in \RR^N$, i.e. it has at most $k$ non-zero entries, from its (potentially noisy) linear measurements $\obs = \A\x + \noise$. Here, $\A\in\RR^{m\times N}$ for $m\ll N$, $\obs \in \RR^m$ and $\noise \in \RR^m$ is a noise vector, typically with a bounded noise level $\|\noise\|_2 \leq \eta$; see \cite{donoho2006compressed,candes2005decoding,candes2006robust}. A well-known result is that, if $\A$ is a random Gaussian matrix, the number of measurements required for most of the classical algorithms like $\ell_1$-minimization or Iterative Hard Thresholding (IHT) to successfully recover the true signal is $m=\bigO\left(k\log\left(N/k\right)\right)$ \cite{candes2006stable,blumensath2009iterative}. 

In model-based compressed sensing we exploit second-order structures beyond the first order sparsity and compressibility structures of a signal to more efficiently encode and more accurately decode the signal. Efficient encoding means taking a fewer number of measurements than in the standard compressed sensing setting, while accurate decoding does not only include smaller recovery error but better interpretability of the recovered solution than in the standard compressed sensing setting. The second order structures of the signal are usually referred to as the structured sparsity of the signal. The idea is, besides standard sparsity, to take into account more complicated structures of the signal \cite{baraniuk2010model}. Nevertheless most of the classical results and algorithms for standard compressed sensing can be adapted to the model-based framework \cite{baraniuk2010model}. 

Numerous applications of model-based compressed sensing exist in practice. 
Key amongst these applications is the multiple measurement vector (MMV) problem, which can be modelled as a block-sparse recovery problem, \cite{eldar2009robust}. The tree-sparse model has been well-exploited in a number of wavelet-based signal processing applications \cite{baraniuk2010model}. In the sparse matrix setting (see Section \ref{sec:prelim}) the model-based compressed sensing was used to solve the Earth Mover Distance problem (EMD). The EMD problem introduced in \cite{schmidt2013constrained} is motivated by the task of reconstructing time sequences of spatially sparse signals, e.g. seismic measurements; see also \cite{hegde2014approximation}. In addition, there are many more potential applications in linear sketching including data streaming \cite{muthukrishnan2005data}, graph sketching \cite{ahn2012graph,gilbert2004compressing}, breaking privacy of databases via
aggregate queries \cite{dwork2007price}, and in sparse regression codes or sparse superposition codes (SPARC) decoding \cite{joseph2014fast,takeishi2014least}, which is also an MMV problem.

Structured sparsity models include tree-sparse, block-sparse, and group-sparse models. For instance, for block-sparse models with dense Gaussian sensing matrices it has been established in \cite{baraniuk2010model} that the number of required measurements to ensure recovery is $m=\bigO(k)$ as opposed to $m=\bigO\left(k\log\left(N/k\right)\right)$ in the standard compressed sensing setting. Furthermore, in the sparse matrix setting, precisely for adjacency matrices of model expander graphs (also known as model expander matrices), the tree-sparse model only requires $m = \bigO\left(\log_k\left(N/k\right)\right)$ measurements \cite{indyk2013model,bah2014model}, which is smaller than the standard compressed sensing sampling rate stated above. Moreover, all proposed algorithms that perform an exact model projection, which is to find the closest vector in the model space for a given signal, guarantee recovery of a solution belonging to the model space, which is then more interpretable than applying off-the-shelf standard compressed sensing algorithms \cite{bah2014model}.

As the exact model projection problem used in many of the classical algorithms may become theoretically and computationally hard for specific sparsity models, approximation variants of some well-known algorithms like the Model-IHT have been introduced in \cite{hegde2015approximation}. Instead of iteratively solving the exact model projection problem, this algorithm, called AM-IHT, uses a head- and a tail-approximation to recover the signal which is computationally less demanding in general. The latter computational benefit comes along with a larger number of measurements required to obtain successful recovery with weaker recovery guarantees: the typical speed versus accuracy trade-off.

A special class of structured sparsity models are group models, where the support of the signal is known to be contained in the union of a small number of groups of indices. Group models were already studied extensively in the literature in the compressed sensing context; see \cite{baraniuk2010low,jenatton2011structured,obozinski2011group}. Its 
choice is motivated by several applications, e.g. in image processing; see \cite{kyrillidis2016convex,rao2012signal,rao2011convex}. As it was shown in \cite{baldassarre2016group} the exact projection problem for group models is \NP-hard in general but can be solved in polynomial time by dynamic programming if the intersection graph of the groups has no cycles. The latter case is quite restricting, since as a consequence each element is contained in at most two groups. In this work we extend existing results for the Model-IHT algorithm and its approximation variant (AM-IHT) derived in \cite{hegde2015approximation} to group models and model expander matrices. We focus on deriving discrete optimization methods to solve the exact projection problem and the head- and tail-approximations for much more general classes of group models than the state-of-the-art. 

In Section \ref{sec:prelim} we present the main preliminary results regarding compressed sensing for structured sparsity models and group models. In Section \ref{sec:exactprojections} we study recovery algorithms using exact projection oracles. We first show that for group models with low treewidth, the projection problem can be solved in polynomial time by dynamic programming which is a generalization of the result in \cite{baldassarre2016group}. 
We then adapt known theoretical results for model expander matrices to these more general group models. 
To solve the exact projection problem for general group models we apply a Benders' decomposition procedure. It can be even used for the more general assumption that we seek a signal which is group-sparse and additionally sparse in the classical sense. In Section \ref{sec:low-freq} we study recovery algorithms using approximation projection oracles, namely head- and tail-approximations. We apply the known results in \cite{hegde2015approximation} to group models of low frequency and show that the required head- and tail-approximations for group models can be solved by a classical greedy-method and LP rounding, respectively. In Section \ref{sec:computations} we test all algorithms, including Model-IHT, AM-IHT, MEIHT and AM-EIHT, on overlapping block-groups and compare the number of required measurements, iterations and the run-time.

\vspace{0.2cm}
{\em Summary of our contributions:}
\vspace{-0.2cm}
\begin{itemize}
\item We study the Model-Expander IHT (MEIHT) algorithm, which was analysed in \cite{bah2014model} for tree-sparse and loopless group-sparse signals, and extend the existing results to general group models, proving convergence of the algorithm.
\item We extend the results in \cite{baldassarre2016group} by proving that the projection problem can be solved in polynomial time if the incidence graph of the underlying group model has bounded treewidth. This includes the case when the intersection graph has bounded treewidth, which generalizes the result for acyclic graphs derived in \cite{baldassarre2016group}. We complement the latter result with a hardness result that we use to justify the bounded treewidth approach.
\item We derive a Benders' decomposition procedure to solve the projection problem for arbitrary group models, assuming no restriction on the frequency or the structure of the groups. The latter procedure even works for the more general model combining group-sparsity with classical sparsity. We integrate the latter procedure into the Model-IHT and MEIHT algorithm.
\item We apply the Approximate-Model IHT (AM-IHT) derived in \cite{hegde2014approximation,hegde2015fast} to Gaussian and expander matrices and to the case of group models with bounded frequency, which is the maximal number of groups an element is contained in. In the expander case we call the algorithm AM-EIHT. To this end we derive both, head- and tail-approximations of arbitrary precision using a classical greedy method and LP-rounding. Using the AM-IHT and the results in \cite{hegde2014approximation,hegde2015fast}, this implies compressive sensing $\ell_2/\ell_2$ recovery guarantees for group-sparse signals.
We show that the number of measurements needed to guarantee a successful recovery exceeds the number needed by the usual model-based compressed sensing bound \cite{baraniuk2010model,blumensath2007sampling} only by a constant factor.
\item We test the algorithms Model-IHT, MEIHT, AM-IHT and AM-EIHT on groups given by overlapping blocks for random signals and measurement matrices. We analyse and compare the minimal number of measurements needed for recovery, the run-time and the number of iterations of the algorithm. 
\end{itemize}

\section{Preliminaries}\label{sec:prelim}

\subsection{Notation}
In most of this work scalars are denoted by ordinary letters (e.g. $x$, $N$), vectors and matrices by boldface letters (e.g. ${\bf x}$, ${\bf A}$), and sets by calligraphic capital letters (e.g., $\mathcal{S}$).
The cardinality of a set $\mathcal{S}$ is denoted by $|\mathcal{S}|$ and we define $[N] := \{1, \ldots, N\}$.
Given $\mathcal{S} \subseteq [N]$, its complement is denoted by $\mathcal{S}^c := [N] \setminus \mathcal{S}$ and $\x_\mathcal{S}$ is the restriction of $\x \in \RR^N$ to $\mathcal{S}$, i.e.~ 
\[(\x_\mathcal{S})_i = \begin{cases} x_i, ~ \mbox{if} ~i \in \mathcal{S},\\ 0 ~ \mbox{otherwise}.\end{cases}
\]
The support of a vector $\x\in\RR^N$ is defined by $\supp (\x)=\left\{ i\in[N] \ | \ x_i\neq 0\right\}$. For a given $k\in\mathbb N$ we say a vector $\x\in\RR^N$ is $k$-sparse if $|\supp (\x)|\le k$.
For a matrix $\A$, the matrix $\A_\support$ denotes a sub-matrix of $\A$ with columns indexed by $\support$.
For a graph $G=(V,E)$ and $\support\subseteq V$, $\Gamma(\support)$ denotes the set of {\em neighbours} of $\support$, that is the set of nodes that are connected by an edge to the nodes in $\support$. We denote by $e_{ij} = (i,j)$ an edge connecting node $i$ to node $j$.
The set $\cG_i$ denotes a group of size $g_i$ and a group model is any subset of $\mathfrak{G} = \{\cG_1,\ldots,\cG_M\}$; while a group model of order $G\in\NN$ is denoted by $\fG_G$, which is a collection of any $G$ groups of $\fG$. For a subset of groups $\mathcal S\subset\mathfrak G$ we sometimes write
\[
\bigcup\mathcal S:=\bigcup_{S\in \mathcal S}S.
\]
 The $\ell_p$ norm of a vector ${\bf x} \in \RR^N$ is defined as 
\[\|{\bf x}\|_p := \left ( \sum_{i=1}^N x_i^p \right )^{1/p}.\]

\subsection{Compressed Sensing}\label{sec:PrelCS}
Recall that the reconstruction problem in standard compressed sensing \cite{donoho2006compressed,candes2005decoding,candes2006robust} attempts to recover an unknown $k$-sparse signal $\x \in \RR^N$, from its (potentially noisy) linear measurements $\obs = \A\x + \noise$, where $\A\in\RR^{m\times N}$, $\obs \in \RR^m$ for $m\ll N$ and $\noise \in \RR^m$ is a noise vector, typically with a bounded noise level $\|\noise\|_2 \leq \eta$. The reconstruction problem can be formulated as the optimization problem
\begin{equation}
\label{eqn:l0prob}
\min_{\x\in\RR^N} ~\|\x\|_0 \quad \mbox{subject to} \quad \|\A\x - \obs\|_2 \leq \eta,
\end{equation}
where $\|\x\|_0$ is the number of non-zero components of $\x$. Problem \eqref{eqn:l0prob} is usually relaxed to an $\ell_1$-minimization problem by replacing $\|\cdot \|_0$ with the $\ell_1$-norm. It has been established that the solution minimizing the $\ell_1$-norm coincides with the optimal solution of \eqref{eqn:l0prob} under certain conditions \cite{candes2006robust}. Besides the latter approach the compressed sensing problem can be solved by a class of greedy algorithms, including the IHT \cite{blumensath2009iterative}. A detailed discussion on compressed sensing algorithms can be found in \cite{foucart2013mathematical}.

The idea behind the IHT can be explained by considering the problem
\begin{equation}
\label{eqn:ihtprob}
\min_{\x\in\RR^N} ~\|\A\x - \obs\|^2_2 \quad \mbox{subject to} \quad  \|\x\|_0 \leq k.
\end{equation}
Under certain choices of $\eta$ and $k$ the latter problem is equivalent to \eqref{eqn:l0prob} \cite{blumensath2009iterative}. Based on the idea of gradient descent methods, \eqref{eqn:ihtprob} can be solved by iteratively taking a gradient descent step, followed by a hard thresholding operation, which sets all components to zero except the largest $k$ in magnitude. Starting with an initial guess $\x^{(0)} = \0$, the $(n+1)$-th IHT update is given by
\begin{equation}
\label{eqn:ihtalgo}
\x^{(n+1)} = \mathcal{H}_k\left[\x^{(n)} + \A^*\left(\A\x^{(n)} - \obs\right)\right],
\end{equation}
where $\mathcal{H}_k:\RR^N\to\RR^N$ is the hard threshold operator and $\A^*$ is the adjoint matrix of $\A$.

Recovery guarantees of algorithms are typically given in terms of what is referred to as the $\ell_p/\ell_q$ {\em instance optimality} \cite{candes2006robust}. Precisely, an algorithm has $\ell_p/\ell_q$ instance optimality if for a given signal $\x$ it always returns a signal $\sol$ with the following error bound
\begin{equation}
\label{eqn:insopt}
\|\x - \sol\|_p \leq c_1(k,p,q) \sigma_k(\x)_q + c_2(k,p,q) \eta,
\end{equation}
where $1\leq q \leq p \leq 2$, $c_1(k,p,q), c_2(k,p,q)$ are constants independent of the dimension of the signal and 
\[\displaystyle \sigma_k(\x)_q = \min_{\z:\| \z\|_0\le k}\|\x-\z\|_q \]
is the {\em best $k$-term approximation} of a signal (in the $\ell_q$-norm). 

Ideally, we would like to have $\ell_2/\ell_1$ instance optimality \cite{candes2006robust}. It turned out that the instance optimality of the known algorithms depends mainly on the sensing matrix $\A$. Key amongst the tools used to analyse the suitability of $\A$ is the restricted isometry property, which is defined in the following.
\begin{definition}[RIP]
	A matrix $\A\in\RR^{m\times N}$ satisfies the $\ell_p$-norm restricted isometry property (RIP-p) of order $k$, with restricted isometry constant (RIC) $\d_k < 1$, if for all $k$-sparse vectors $\x$
	\begin{equation}
	\label{eqn:rip}
	\left(1-\d_k\right)\|\x\|_p^p \leq \|\A\x\|_p^p \leq \left(1+\d_k\right)\|\x\|_p^p.
	\end{equation}
\end{definition}
Typically RIP without the subscript $p$ refers to case when $p=2$. We use this general definition here because we will study the case $p=1$ later. The RIP is a sufficient condition on $\A$ that guarantees optimal recovery of $\x$ for most of the known algorithms. If the entries of $\A$ are drawn i.i.d from a sub-Gaussian distribution and $m = \bigO\left(k\log(N/k)\right)$, then $\A$ has RIP-2 with high probability and leads to the ideal $\ell_2/\ell_1$ instance optimality for most algorithms; see \cite{candes2006stable}. Note that the bound $\bigO\left(k\log(N/k)\right)$ is asymptotically tight. On the other hand, deterministic constructions of $\A$ or random $\A$ with binary entries with non-zero mean do not achieve this optimal $m$, and are faced with the so-called {\em square root bottleneck} where $m=\Omega\left(k^2\right)$; see  \cite{devore2007deterministic,chandar2008negative}.

\subsubsection*{Sparse Sensing Matrices from Expander Graphs.} ~The computational benefits of sparse sensing matrices necessitated finding a way to circumvent the square root bottleneck for non-zero mean binary matrices. One such class of binary matrices is the class of adjacency matrices of expander graphs (henceforth referred to as {\em expander matrices}), which satisfy the weaker RIP-1. Expander graphs are objects of interest in pure mathematics and theoretical computer science, for a detailed discourse on this subject see \cite{hoory2006expander}. We define an expander graph as follows: 
\begin{definition}[Expander graph]
	\label{def:llexpander}
	Let $H=\left( [N],[m],\edges \right)$ be a left-regular bipartite graph with $N$ left vertices, $m$ right vertices, a set of edges $\edges$ and left degree $d$.
	If for any $\epsilon \in (0,1/2)$ and any $\support \subset [N]$ of size $|\support|\leq k$ it holds $|\Gamma(\support)| \geq (1-\epsilon)dk$, then $H$ is referred to as a {\em $(k,d,\epsilon)$-expander graph}.
\end{definition}
An expander matrix is the adjacency matrix of an expander graph. Choosing $m = \bigO\left(k\log(N/k)\right)$, then a random bipartite graph $H=\left( [N],[m],\edges \right)$ with left degree $d=\bigO\left( \frac{1}{\epsilon}\log(\frac{N}{k})\right)$ is an $(k,d,\epsilon )$-expander graph with high probability \cite{foucart2013mathematical}. Furthermore expander matrices achieve the sub-optimal $\ell_1/\ell_1$ instance optimality \cite{berinde2008combining}. For completeness we state the lemma in \cite{jafarpour2009efficient} deriving the RIC for such matrices. 
\begin{lemma}[RIP-1 for Expander Matrices, \cite{jafarpour2009efficient}]\label{lem:rip1}
	Let $\A$ be the adjacency matrix of a $(k,d,\epsilon)$-expander graph $H$, then for any $k$-sparse vector $\x$, we have
	\begin{equation}
	\label{eqn:rip1}
	\left(1-2\epsilon\right)d\|\x\|_1 \leq \|\A\x\|_1 \leq d\|\x\|_1.
	\end{equation}
\end{lemma}

The most relevant algorithm that exploits the structure of the expander matrices is the Expander-IHT (EIHT) proposed in \cite{foucart2013mathematical}. Similar to the IHT algorithm it performs updates 
\begin{equation}
\label{eqn:eihtalgo}
\x^{(n+1)} = \mathcal{H}_k\left[\x^{(n)} + \med\left(\A\x^{(n)} - \obs\right)\right],
\end{equation}
where $\med: \RR^m \rightarrow \RR^N$ is the {\em median operator} and $[\med(\z)]_i = \mbox{median}\left(\{z_j\}_{j\in \Gm(i)}\right)$ for each $\z\in\RR^m$. For expander matrices the EIHT achieves $\ell_1/\ell_1$ instance optimality \cite{foucart2013mathematical}. 

\subsection{Model-based Compressed Sensing}\label{sec:PrelimModelBasedCS}
Besides sparsity (and compressibility) signals do exhibit more complicated structures. When compressed sensing takes into account these more complicated structures (or models) in addition to sparsity, it is usually referred to as model-based compressed sensing or {\em structured sparse recovery} \cite{baraniuk2010model}. A precise definition is given in the following:
\begin{definition}[Structured Sparsity Model \cite{baraniuk2010model}]
A \textit{structured sparsity model} is a collection of sets, $\mathfrak M=\left\{ \mathcal S_1 ,\ldots , \mathcal S_M\right\}$ with $|\mathfrak M| = M$, of allowed structured supports $\mathcal S_i\subseteq [N]$.
\end{definition}
Note that the classical $k$-sparsity studied in Section \ref{sec:PrelCS} is a special case of a structured sparsity model where all supports of size at most $k$ are allowed. Popular structured sparsity models include tree-sparse, block-sparse, and group-sparse models \cite{baraniuk2010model}. In this work we study group-sparse models which we will introduce in Section \ref{sec:Prel_GroupStructures}.

Similar to the classical sparsity case the RIP property is defined for structured sparsity models.
\begin{definition}[Model-RIP \cite{baraniuk2010model}]
A matrix $\A\in\RR^{m\times N}$ satisfies the $\ell_p$-norm model restricted isometry property ($\mathfrak{M}$-RIP-$p$) with model restricted isometry constant ($\mathfrak{M}$-RIC) $\d_{\mathfrak M} < 1$, if for all vectors $\x$ with $\supp (\x)\in \mathfrak M$ it holds
	\begin{equation}
	\label{eqn:model-rip}
	\left(1-\d_{\mathfrak M}\right)\|\x\|_p^p \leq \|\A\x\|_p^p \leq \left(1+\d_{\mathfrak M}\right)\|\x\|_p^p.
	\end{equation}
\end{definition}
In \cite{baraniuk2010model} it was shown that for a matrix $\A\in\RR^{m\times N}$ to have the Model-RIP with high probability the required number of measurements is $m = \bigO(k)$ for tree-sparse signals and $m = \bigO\left(kg + \log\left(N/(kg)\right)\right)$ for block-sparse signals with block size $g$, when the sensing matrices are dense (typically sub-Gaussian). In general for a given structured sparsity model $\mathfrak M$ for sub-Gaussian random matrices the number of required measurements is $m=\bigO\left(\d_{\mathfrak M}^{-2}g\log (\d_{\mathfrak M}^{-1}) + \log(|\mathfrak M |)\right)$ where $g$ is the cardinality of the largest support in $\mathfrak M$.
 
Furthermore the authors in \cite{baraniuk2010model} show that classical algorithms like the IHT can be modified for structured sparsity models to achieve instance optimality. To this end the hard thresholding operator $\mathcal{H}_k$ used in the classical IHT is replaced by a model-projection oracle which for a given signal $\x\in\RR^N$ returns the closest signal over all signals having support in $\mathfrak M$. We define the model-projection oracle in the following.
\begin{definition}[Model-Projection Oracle \cite{baraniuk2010model}]\label{def:projection_oracle}
Given $p\ge 1$, a \textit{model-projection oracle} is a function $\mathcal P_{\mathfrak M}:\RR^N \to \RR^N$ such that for all $\x\in\RR^N$ we have $\supp (\mathcal P_{\mathfrak M}(\x))\in \mathfrak M$ and it holds
\[
\| \x - \mathcal P_{\mathfrak M}(\x )\|_p = \min_{\support\in\mathfrak M} \| \x - \x_{\support}\|_p . 
\]
\end{definition}
From the definition it directly follows that $\mathcal P_{\mathfrak M}(\x )_i = \x_i$ if $i\in\supp (\mathcal P_{\mathfrak M}(\x))$ and $0$ otherwise. Note that in the case of classical $k$-sparsity the model-projection oracle is given by the hard thresholding operator $\mathcal{H}_k$. In contrast to this case, calculating the optimal model projection $P_{\mathfrak M} (\x)$ for a given signal $\x\in\RR^N$ and a given structured sparsity model $\mathfrak M$ may be computationally hard. Depending on the model $\mathfrak M$ finding the optimal model projection vector may be even NP-hard; see Section \ref{sec:Prel_GroupStructures}. The modified version of the IHT derived in \cite{baraniuk2010model} is presented in Algorithm \ref{alg:Model-IHT}.

\begin{algorithm}\caption{(Model-IHT)}\label{alg:Model-IHT}
	\inp $\A$, $\obs$, $\noise$ with $\obs = \A\x +\noise$ and $\mathfrak M$\\
	\out $\text{Model-IHT}(\obs,\A,\mathfrak{M})$, an $\mathfrak M$-sparse approximation to $\x$
	\begin{algorithmic}[1]
		\State $\x^{(0)} \gets \boldsymbol{0}$; \ $n\gets 0$
		\While {halting criterion false}  
		\State $\x^{(n+1)} \gets \mathcal{P}_{\mathfrak M}\left[\x^{(n)} + \A^*(\obs-\A\x^{(n)})\right]$
		\State $n\gets n+1$
		\EndWhile
		\State Return: $\text{Model-IHT}(\obs,\A,\mathfrak{M}) \gets \x^{(n)}$
	\end{algorithmic}
\end{algorithm}
\vspace{-0.3cm}
\noindent Note that common halting criterion is given by a maximum number of iterations or a bound on the iteration error $\|\x^{(n+1)}-\x^{(n)}\|_p$.

\subsubsection*{Model-Sparse Sensing Matrices from Expander Graphs.}
In the sparse matrix setting the sparse matrices we consider as {\em model expander matrices}, which are adjacency matrices of model expander graphs, defined thus.
	\begin{definition}[Model expander graph]
		\label{def:mdlexpander}
		Let $H=\left( [N],[m],\edges \right)$ be a left-regular bipartite graph with $N$ left vertices, $m$ right vertices, a set of edges $\edges$ and left degree $d$.
		Given a model $\mathfrak M$, if for any $\epsilon_{\mathfrak M} \in (0,1/2)$ and any $\support = \cup_{\support_i \in \mathcal K}\support_i$, with $\mathcal K \subset \mathfrak M$ and $|\support| \leq s$, we have $|\Gamma(\support)| \geq (1-\epsilon_{\mathfrak M})d|\support|$, then $H$ is referred to as a {\em $(s,d,\epsilon_{\mathfrak M})$-model expander graph}.
	\end{definition}

In this setting the known results are sub-optimal. Using {\em model expander matrices} for tree-sparse models the required number of measurements to obtain instance optimality is $m = k\log\left(N/k\right)/\log\log\left(N/k\right)$ which was shown in \cite{indyk2013model,bah2014model}.

A key ingredient in the analysis for the afore-mentioned sample complexity results for model expanders is the {\em model}-RIP-1, which is just RIP-1 for model expander matrices (hence they are also called model-RIP-1 matrices \cite{indyk2013model}). Consequently, Lemma \ref{lem:rip1} also holds for these model-RIP-1 matrices \cite{indyk2013model}.


First, in \cite{bah2014model} the {\em Model Expander IHT} (MEIHT) was studied for loopless overlapping groups and $D$-ary tree models. Similar to Algorithm \ref{alg:Model-IHT} the MEIHT is a modification of EIHT where the hard threshold operator $\mathcal{H}_k$ is replaced by the projection oracle $\mathcal{P}_{\mathfrak M}$ onto the model $\mathfrak M$. Thus the update of the MEIHT in each iteration is given by
\begin{equation}
\label{eqn:meihtalgo}
\x^{(n+1)} = \mathcal{P}_{\mathfrak M}\left[\x^{(n)} + \med\left(\A\x^{(n)} - \obs\right)\right].
\end{equation}
In \cite{bah2014model} the authors show that this algorithm always returns a solution in the model, which is highly desirable for some applications. The running time is given in the proposition below.
\begin{proposition}[Proposition 3.1, \cite{bah2014model}]
	\label{pro:meiht}
	The runtime of MEIHT is $\bigO\left(kN\bar n\right)$ and $\bigO\left(M^2G\bar n + N\bar n\right)$ for $D$-ary tree models and loopless overlapping group models respectively,  where $k$ is the sparsity of the tree model and $G$ is the number of active groups (i.e. group sparsity of the model), $\bar n$ is  the  number  of iterations, $M$ is the number of groups and $N$ is the dimension of the signal.
\end{proposition}
%
%

\subsection{Group Models}\label{sec:Prel_GroupStructures}
The models of interest in this paper are group models. A \emph{group model} is a collection $\mathfrak{G} = \{\cG_1,\ldots,\cG_M\}$ of \emph{groups} of indices, i.e. $\cG_i\subset [N]$, together with a \emph{budget} $G \in [M]$. We denote $\mathfrak G_G$ as the structured sparsity model (i.e. group-sparse model) which contains all supports contained in the union of at most $G$ groups in $\mathfrak G$, i.e.  
\begin{equation}
\label{eqn:G-sparse-model}
\mathfrak G_G:=\left\{\support\subseteq \NN \ | \ \support\subseteq\bigcup_{i\in \mathcal I}\cG_i, ~|\mathcal I|\le G \right\}.
\end{equation}
We will always tacitly assume that $\bigcup_{i=1}^{M} \cG_i = [N]$. We say that a signal $\mathbf x \in \mathbb R^N$ is \emph{$G$-group-sparse} if the support of $\mathbf x$ is contained in $\mathfrak G_G$. If $G$ is clear from the context, we simply say that $\mathbf x$ is \emph{group-sparse}. Let $g_i = |\cG_i|$ and denote $g_{\text{max}} = \max_{i\in [M]} g_i$ as the size of largest support in $\mathfrak G_G$.
The \textit{intersection graph} of a group model is the graph which has a node for each group $\mathcal G_i\in \mathfrak G$ and has an edge between $\mathcal G_i$ and $\mathcal G_j$ if the groups overlap, i.e. if $\mathcal G_i\cap \mathcal G_j \neq \emptyset$; see \cite{baldassarre2016group}. We call a group model \textit{loopless} if the intersection graph of the group model has no cycles. We call a group model \textit{block model} if all groups have equal size and if they are pairwise disjoint. In this case the groups are sometimes called blocks. We define the \textit{frequency} $f$ of a group model as the maximum number of groups an element is contained in, i.e.
\[
f:=\max_{i\in [N]} |\left\{ j\in [M] \ | \ i\in \mathcal G_j \right\}| .
\]

In \cite{baldassarre2016group} besides the latter group models, the more general models are considered where an additional sparsity in the classical sense is required on the signal. More precisely for a given budget $G\in [M]$ and a sparsity $K\in [N]$ they study the structured sparsity model
\begin{equation}
\label{eqn:GK-sparse-model}
\mathfrak G_{G,K}:=\left\{\support\subseteq \NN \ | \ \support\subseteq\bigcup_{i\in \mathcal I}\cG_i, ~|\mathcal I|\le G, \ |\support|\le K \right\}.
\end{equation}
Note that for $K=N$ we obtain a standard group model defined as above.

Both variants of group models defined above clearly are special cases of a structured sparsity model defined in Section \ref{sec:PrelimModelBasedCS}. Therefore all results for structured sparsity models can be used for group-sparse models. To adapt Algorithm \ref{alg:Model-IHT} a model projection oracle $\mathcal P_{\mathfrak G_G}$ (or $\mathcal P_{\mathfrak G_{G,K}}$) has to be provided. Note that for several applications we are not only interested in the optimal support of the latter projection but we want to find at most $G$ groups covering this support. The main work of this paper is to analyse the complexity of the latter problem for group models and to provide efficient algorithms to solve it exactly or approximately. Given a signal $\x \in \mathbb R^N$, the \MS~problem or sometimes called \textit{signal approximation problem} is then to find a support $\support\in\mathfrak G_{G,K}$ together with $G$ groups covering this support such that $||\x-\x_{\support}||_p$ is minimal, i.e. we want to solve the problem
\[
\min_{\substack{\mathcal G_1,\ldots,\mathcal G_G\in \mathfrak G \\ \support\subset \bigcup_{i=1}^{G}\mathcal G_i \\ |\support |\le K}} \| \x - \x_{\support}\|_p .
\]
If the parameter $K$ is not mentioned we assume $K=N$.

Baldassarre et al.~\cite{baldassarre2016group} observed the following close relation to the NP-hard \MWC~problem.
Given a signal $\mathbf x \in \mathbb R^N$, a group-sparse vector $\mathbf{\hat x}$ for which $||\mathbf x-\mathbf{\hat x}||_2^2$ is minimum satisfies $\hat x_i \in \{0,x_i\}$ for all $i \in [N]$.
For a vector with the latter property,
\[
||\mathbf x-\mathbf{\hat x}||_2^2 = \sum_{i=1}^N x_i^2 - \sum_{i=1}^N \hat x_i^2
\]
holds and so minimizing $||\mathbf x-\mathbf{\hat x}||_2^2$ is equivalent to maximizing $\sum_{i=1}^N \hat x_i^2$.
Consequently, the \MS~problem with $K=N$ is equivalent to the problem of finding an index set $\mathcal I\subset [M]$ of at most $G$ groups, i.e. $|\mathcal I| \le G$, maximizing $\sum_{i\in\bigcup_{j\in \mathcal I} \mathcal G_j} x_i^2$.
This problem is called \MWC~in the literature \cite{hochbaum1998analysis}.
Despite the prominence of the latter problem, we will stick to the group model notation, since it is closer to the applications we have in mind and we will leave the regime of \MWC~by introducing more constraints later.

We simplify the notation by defining $w_i = x_i^2$ for all $i \in [N]$. Using this notation, the \MS~problem is equivalent to finding an optimal solution of the following integer program:
\begin{equation}\label{eqn:generalprob-IP}
\begin{aligned}
\max \ & \mathbf w^\top \bu \\
s.t. \quad & \sum_{i \in [M]} v_i \le G \\
& u_j \le \sum_{i:j \in \cG_i} v_i \quad \mbox{for all } j \in [N] \\
& \mathbf \bu\in \{0,1\}^N, \ \mathbf \bv\in \{0,1\}^M
\end{aligned}
\end{equation}
Here, the variable $u_i$ is one if and only if the $i$-th index is contained in the support of the signal approximation, and $v_i$ is one if and only if the group $\cG_i$ is chosen.

Besides the NP-hardness for the general case the authors in \cite{baldassarre2016group} show that the \MS~problem can be solved in polynomial time by dynamic programming for the special case of loopless groups. Furthermore the authors show that if the intersection graph is bipartite the projection problem can be solved in polynomial time by relaxing problem \eqref{eqn:generalprob-IP}. Similar results are obtained for the more general problem, where additional to the group-sparsity the classical $K$-sparsity is assumed, i.e. the additional constraint 
\[\sum_{i\in [N]}u_i\le K\]
is added to problem \eqref{eqn:generalprob-IP}.

As stated in Section \ref{sec:PrelimModelBasedCS}, the authors of \cite{baraniuk2010model} first study a special case of group models, i.e. block models, where the groups are non-overlapping and are all of equal size. The sample complexity they derived in that work for sub-Gaussian measurement matrices is $m = \bigO\left(Gg + \log\left(N/(Gg)\right)\right)$, where $g$ is the fixed block size. 
However, in \cite{bah2014model} the authors studied group models in the sparse matrix setting, besides other results they proposed the MEIHT algorithm for tree and group models. The more relevant result to this work show that for {\em loopless} overlapping group-sparse models with maximum group size $g_{\text{max}}$, using model expander measurement matrices, the number of measurements required for successful recovery is $m = Gg_{\text{max}}\log\left(N/(Gg_{\text{max}})\right)/\log\left(Gg_{\text{max}}\right)$; see \cite{bah2014model}. This results holds for general groups, the ``{\em looplessness}'' condition is only necessary for the polynomial time reconstruction using the MEIHT algorithm. Therefore, this sample complexity result also holds for the general group models we consider in this manuscript. 

\subsubsection{Group Lasso}
The classical Lasso approach for $k$-sparse signals seeks to minimize a quadratic error penalized by the $\ell_1$-norm \cite{foucart2013mathematical}. More precisely, for a given $\lambda>0$ we want to find an optimal solution of the problem
\begin{equation}\label{eqn:lasso}
\min_{\x} \|A\x-\obs\|_2^2 + \lambda\|\x\|_1 .
\end{equation}
It is well known that using the latter approach for appropriate choices of $\lambda$ leads to sparse solutions.

The Lasso approach was already extended to group models in \cite{yuan2006model} and afterwards studied in several works for non-overlapping groups; see \cite{huang2010benefit,kolar2011union,lounici2011oracle,negahban2011simultaneous}. The idea is again to minimize a loss function, e.g. the quadratic loss, and to penalize the objective value for each group by a norm of the weights of the recovered vector restricted to the items in each group. An extension which can also handle overlapping groups was studied in \cite{zhao2006grouped,jenatton2011proximal}. In \cite{obozinski2011group} the authors study what they call the \textit{latent group Lasso}. To this end they consider a loss function $L:\RR^N\to \RR$ and propose to solve the $\ell_1 /\ell_2$-penalized problem
\begin{equation}\label{eqn:grouplassooverlap}
\begin{aligned}
\min \ & L(\x) + \lambda \sum_{\mathcal G\in\mathfrak G} d_{\mathcal G} \| \w^{\mathcal G}\|_2 \\
s.t. \quad & \x=\sum_{\mathcal G\in\mathfrak G}\w^{\mathcal G}\\
& \supp\left( \w^{\mathcal G}\right)\subset \mathcal G \ \forall\mathcal G\in\mathfrak G\\
& \x\in\RR^N, \w^{\mathcal G}\in\RR^N \ \forall \mathcal G\in\mathfrak G
\end{aligned}
\end{equation}
for given weights $\lambda>0$ and $d_{\mathcal G}\ge 0$ for each $\mathcal G\in \mathfrak G$. The idea is that for ideal choices of the latter weights a solution of Problem \eqref{eqn:grouplassooverlap} will be sparse and its support is likely to be a union of groups. Nevertheless it is not guaranteed that the number of selected groups is optimal as it is the case for the iterative methods in the previous sections. Note that equivalently we can replace each norm $\|\w^{\mathcal G}\|$ by a variable $z^{\mathcal G}$ in the objective function and add the quadratic constraint $\| \w^{\mathcal G}\|_2^2\le (z^{\mathcal G})^2$. Hence Problem \eqref{eqn:grouplassooverlap} can be modelled as a quadratic problem and can be solved by standard solvers like CPLEX.

The $l_0$ counterpart of Problem \eqref{eqn:grouplassooverlap} was considered in \cite{huang2011learning} under the name \textit{block coding} and can be formulated as
\begin{equation}\label{eqn:grouplassoL0}
\begin{aligned}
\min \ & L(\x) + \lambda \sum_{\mathcal G\in\tilde{\mathfrak G}} d_{\mathcal G} \| \w^{\mathcal G}\|_2 \\
s.t. \quad & \x=\sum_{\mathcal G\in\tilde{\mathfrak G}}\w^{\mathcal G}\\
& \supp\left(\w^{\mathcal G}\right)\subset \mathcal G \ \forall\mathcal G\in\tilde{\mathfrak G}\\
&\tilde{\mathfrak G}\subset\mathfrak G\\
& \x\in\RR^N, \w^{\mathcal G}\in\RR^N \ \forall \mathcal G\in\tilde{\mathfrak G} .
\end{aligned}
\end{equation}
Note that in contrast to Problem \eqref{eqn:grouplassooverlap} an easy reformulation of Problem \eqref{eqn:grouplassoL0} into a continuous quadratic problem is not possible. Nevertheless we can reformulate it using the mixed-integer programming formulation
\begin{equation}\label{eqn:grouplassoL0IPFormulation}
\begin{aligned}
\min \ & L(\x) + \lambda \sum_{\mathcal G\in\mathfrak G} d_{\mathcal G} \bv^{\mathcal G} \\
s.t. \quad & \x=\sum_{\mathcal G\in\mathfrak G}\w^{\mathcal G}\\
& \supp\left(\w^{\mathcal G}\right)\subset \mathcal G \ \forall\mathcal G\in\mathfrak G\\
&w_i^{\mathcal G}\le M_i\bv^{\mathcal G} \ \forall i\in[N], \mathcal G\in\mathfrak G\\
&-M_i\bv^{\mathcal G}\le \w_i^{\mathcal G}\ \forall i\in[N], \mathcal G\in\mathfrak G\\ 
& \x\in\RR^N, \w^{\mathcal G}\in\RR^N, \bv^{\mathcal G}\in\left\{ 0,1\right\} \ \forall \mathcal G\in\mathfrak G
\end{aligned}
\end{equation}
where $M_i\in\RR$ can be chosen larger or equal to the entry $|x_i|$ of the true signal for each $i\in [N]$. The variables $\bv^{\mathcal G}\in\left\{ 0,1\right\}$ have value $1$ if and only if group $\mathcal G$ is selected for the support of $\x$. As for the $\ell_1 / \ell_2$ variant it is not guaranteed that the number of selected groups is optimal. Note that the latter problem is a mixed-integer problem and therefore hard to solve in large dimension in general. Furthermore the efficiency of classical methods as the branch \& bound algorithm depend on the quality of the calculated lower bound which depends on the values $M_i$. Hence in practical applications where the true signal is not known good estimations of the $M_i$ values are crucial for the success of the latter method. Another drawback is that the best values for $\lambda$ and the weights $d_{\mathcal G}$ are not known in advance and have to be chosen appropriately for each application.

We study Problems \eqref{eqn:grouplassooverlap} and \eqref{eqn:grouplassoL0} computationally in Section \ref{sec:computations}.

\subsection{Approximation Algorithms for Model-Based Compressed Sensing}\label{sec:preliminaries_AMIHT}
As mentioned in the last section solving the projection problem, given in Definition \ref{def:projection_oracle}, may be computationally hard. To overcome this problem the authors in \cite{hegde2015fast,hegde2015approximation} present algorithms, based on the idea of IHT (and CoSaMP), which instead of solving the projection problems exactly, use two approximation procedures called head- and tail-approximation. In this section we will shortly describe the concept and the results in \cite{hegde2015fast,hegde2015approximation}. Note that we just present results related to the IHT, although similar results for the CoSaMP were derived as well in \cite{hegde2015fast,hegde2015approximation}.

Given two structured sparsity models $\mathfrak M, \mathfrak M_H$ and a vector $\x$, let $\mathcal H$ be an algorithm that computes a vector $\mathcal H(\x)$ with support in $\mathfrak M_H$. Then, given some $\alpha \in \mathbb R$ (typically $\alpha < 1$) we say that $\mathcal H$ is an \emph{$(\alpha,\mathfrak M, \mathfrak M_H,p)$-head approximation} ~if
\begin{equation}\label{eqn:head-condition}
|| \mathcal H(\x)||_p \ge \alpha \cdot || \x_{\mathcal S}||_p \mbox{ for all } \mathcal S \in \mathfrak M.
\end{equation}
Note that the support of the vector calculated by $\mathcal H$ is contained in $\mathfrak M_H$ while the approximation guarantee must be fulfilled for all supports in $\mathfrak M$.

Moreover given two structured sparsity models $\mathfrak M, \mathfrak M_T$  let $\mathcal T$ be an algorithm which computes a vector $\mathcal T(\x)$ with support in $\mathfrak M_T$. Given some $\beta \in \mathbb R$ (typically $\beta > 1$) we say that $\mathcal T$ is a \emph{$(\beta,\mathfrak M, \mathfrak M_T,p)$-tail approximation} if
\begin{equation}\label{eqn:tail-condition}
|| \x - \mathcal T(\x) ||_p \le \beta \cdot || \x - \x_{\mathcal S}||_p \mbox{ for all } \mathcal S \in \mathfrak M.
\end{equation}
Note that in general a head approximation does not need to be a tail approximation and vice versa.

The cases studied in \cite{hegde2015approximation} are  $p=1$ and $p=2$. For the case $p=2$ the authors propose an algorithm called \textit{Approximate Model-IHT} (AM-IHT), shown in Algorithm \ref{alg:am-iht}.

\begin{algorithm}\caption{(AM-IHT)}\label{alg:am-iht}
\inp $\A$, $\obs$, $\noise$ with $\A\mathbf x=\obs+\noise$ and $\mathfrak M$\\
\out $\text{AM-IHT}(\obs,\A,\mathfrak M)$, an approximation to $\x$
\begin{algorithmic}[1]
\State $\x^0 \gets 0$; $n\gets 0$
\While{halting criterion is false} 
\State $\x^{(n+1)} \gets \mathcal T(\x^{(n)} + \mathcal H(\A^*(\obs-\A\x^{(n)})))$ \label{step:head_tail_step}
\State $n\gets n+1$
\EndWhile
\State Return: $\text{AM-IHT}(\obs,\A,\mathfrak M) \gets \x^{(n)}$
\end{algorithmic}
\end{algorithm}

Assume that $\mathcal T$ is a $(\beta,\mathfrak M, \mathfrak M_T,2)$-tail approximation and $\mathcal H$ is a $(\alpha,\mathfrak M_T\oplus\mathfrak M, \mathfrak M_H,2)$-head approximation where $\mathfrak M_T\oplus\mathfrak M$ is the Minkowski sum of $\mathfrak M_T$ and $\mathfrak M$. Furthermore we assume the condition
\begin{equation}\label{eqn:condition_convergenceofAMIHT}
\alpha^2 > 1-(1+\beta)^{-2}
\end{equation}
holds. The authors in \cite{hegde2015approximation} prove that for a signal $\x\in\RR^N$ with $\supp (\x )\in\mathfrak M$, noisy measurements $\obs=\A\x+\noise$ where $\A$ has $\mathfrak M \oplus \mathfrak M_T\oplus\mathfrak M_H$-model RIP with RIC $\delta$, 
Algorithm \ref{alg:am-iht} calculates a signal estimate $\hat{\x}$ satisfying 
\[
\| \x-\hat{\x} \|_2 \le \tau\|\noise\|_2
\]
where $\tau$ depends on $\delta, ~\alpha$ and $\beta$. Note that the condition \eqref{eqn:condition_convergenceofAMIHT} holds e.g. for approximation accuracies $\alpha>0.9$ and $\beta<1.1$.

For the case $p=1$ the authors replace Step \ref{step:head_tail_step} in Algorithm \ref{alg:am-iht} by the update
\[
\x^{(n+1)} \gets \mathcal T(\x^{(n)} + \mathcal H(\med(\obs-\A\x^{(n)})))
\]
where $\med$ is the median operator defined as in Section \ref{sec:PrelCS}. Under the same assumptions as above, but considering $p=1$ for the head- and tail-approximations and $\A$ having the $\mathfrak M$-RIP-$1$, the authors in \cite{hegde2015approximation} show convergence of the adapted algorithm.\\

\subsection{Comparison to Related Works}\label{sec:comparison}
In a more illustrative way in Tables \ref{tab:model-encoding} and \ref{tab:model-decoding} below we show where our results stand vis-a-vis other results. In Table \ref{tab:model-encoding} we show the studied models together with the derived sample complexity and the studied class of measurements matrices. In Table \ref{tab:model-decoding} we present the names of the studied algorithms, the class of the model projection, the class of the algorithm used to solve the model projection, the runtime complexity of the projection problem and the class of instance optimality.

\begin{table}[h]
	\centering
	\caption{Comparison of model-based compressive sensing results. This table shows the studied models and the sensing framework, particularly showing sample complexities for the various models together with the appropriate measurement matrices (referred as ensemble). Here $N$ is the dimension of the signal, $k$ is the standard sparsity of a signal, $G$ is the number of active groups (also referred to as group/block sparsity), $g$ is the block size, $g_{\max}$ is the size of the largest group, $s$ is the size of the support of a forest (union of sub-graphs) in a given graphs, $c$ is the maximum number of connected components in the forest, $B$ is a bound on the total weights of edges in the forest, and $\rho(H)$ is the  weight-degree  of a graph $H$, see respective references of the complexities for further explanation of these terms.
	}
	\begin{tabular}{|c||c||c|c|}
		\hline
		\multirow{3}{*}{\bf Paper}		& {\bf Signal}  & \multicolumn{2}{c|}{\bf Sensing}\\
		\cline{2-4}
										& \multirow{2}{*}{\bf Model}	& {\bf Sample} 			& \multirow{2}{*}{\bf Ensemble} \\
										&								&	{\bf complexity}	&  	\\
		\hline
		\hline
		{\cite{obozinski2011group}}					& Groups with overlaps		& {--}	 														& {--}			\\
		\hline
		\cite{eldar2009robust}						& Blocks		& $G\log\left(\frac{N}{Gg}\right)$											&	Gaussian	 \\
		\hline
		\multirow{2}{*}{\cite{baraniuk2010model}}	& Binary trees	& $k$ 																		&	\multirow{2}{*}{Subgaussian}\\
													& Blocks		& $Gg + G\log\left(\frac{N}{G}\right)$ 									&		 	 \\
		\hline
		\cite{schmidt2013constrained}	   			& Constrained earth mover distance & --								&	--	\\
		\hline
		\multirow{2}{*}{\cite{indyk2013model}}		& Binary Trees	& $\frac{k\log\left(\frac{N}{k}\right)}{\log\log\left(\frac{N}{k}\right)}$ 	&	\multirow{2}{*}{Model RIP-1 matrices}	\\
													& Blocks				& $\frac{Gg\log(N)}{\log(Gg)}$ 										&	 \\
		\hline
		\multirow{2}{*}{\cite{bah2014model}}		& $D$-ary Trees & $\frac{k\log\left(\frac{N}{k}\right)}{\log\log\left(\frac{N}{k}\right)}$	&	\multirow{2}{*}{Expander}  \\
													& Groups with loopless overlaps  	& $\frac{Gg_{\max}\log(N)}{\log\left(Gg_{\max}\right)}$						&	  \\
		\hline
		\cite{hegde2015nearly}	   					& Weighted graph model & $s\left(\log \left(\rho(H)\right) + \log\left(\frac{B}{s}\right)\right)+c\log\left(\frac{N}{c}\right)$								&	Gaussian	\\
		\hline
		\cite{kyrillidis2016convex}	   				& Blocks		& $Gg+G\log\left(\frac{N}{G}\right)$											&	Expander	\\
		\hline
		This   										& \multirow{2}{*}{Groups with overlaps} 	& \multirow{2}{*}{$\frac{Gg_{\max}\log(N)}{\log\left(Gg_{\max}\right)}$}	& {Subgaussian} \\
		work 										&  &     																		&	\& {Expander}	\\
		\hline
	\end{tabular}
	\label{tab:model-encoding}
\end{table}

\begin{table}[h]
	\centering
	\caption{Comparison of model-based compressive sensing results. This table shows results about the reconstruction approaches used for the models stated in Table \ref{tab:model-encoding}. Here $M$ is the number of groups, $\bar n$ is the number of iterations the algorithm takes to get a given solution, $\ell_{21}$ refers to the mixed $\ell_1$ and $\ell_2$ group norm, and $w_T$ is the treewidth of the incidence graph. Abbreviations: SOCP - second order cone program, DP - dynamic programming, BD - Bender's decomposition, CD - covariate  duplication, and BCD - block-coordinate descent. For space purposes we neglect the {\em big-O} notation in the complexities. For algorithms whose runtimes are not explicitly stated but implied to be either {\em polynomial} or {\em exponential} in their respective references, we just state this without the explicit expression. 
	}
	\begin{tabular}{|c||c|c|c|c|c|}
		\hline
		\multirow{3}{*}{\bf Paper}					& \multicolumn{5}{c|}{\bf Reconstruction} \\
		\cline{2-6}
													& {\bf Algorithm} 	& {\bf Model} 		& {\bf Projection}& {\bf Runtime} 	 	& {\bf Instance}\\
													& {\bf name} 		& {\bf projection} 	& {\bf algorithm}				& {\bf complexity} 	 		& {\bf optimality}\\
		\hline
		\hline
		{\cite{obozinski2011group}}					& Group-LASSO		& {Exact}   	& {CD \& BCD}			& {Polynomial}				& {--}\\
		\hline
		\cite{eldar2009robust}						& $\ell_{21}$-minimization 	& Exact 		& {SOCP}			& Polynomial	 				& $\ell_2/\ell_{21}$\\
		\hline
		\multirow{2}{*}{\cite{baraniuk2010model}}	& \multirow{2}{*}{CoSaMP \& IHT}	& \multirow{2}{*}{Approximate} 	& \multirow{2}{*}{--} 	& \multirow{2}{*}{$N\log^2(N)$}	 & \multirow{2}{*}{--}\\
		 											&									&  								&    & 			  					& \\
		\hline
		\multirow{2}{*}{\cite{schmidt2013constrained}}	   			& EMD-IHT			& \multirow{2}{*}{Exact} 	& \multirow{2}{*}{--} 	& \multirow{2}{*}{Polynomial}	 		& \multirow{2}{*}{--}\\
			   										& EMD-CoSaMP		&  							&  	& 	& \\
		\hline
		\multirow{2}{*}{\cite{indyk2013model}}		& Model-based 		& \multirow{2}{*}{Approximate}	& \multirow{2}{*}{--}	& \multirow{2}{*}{Exponential} 	& \multirow{2}{*}{$\ell_1/\ell_1$}\\
													& sparse recovery 	&  								& 	&   						& \\
		\hline
		\multirow{2}{*}{\cite{bah2014model}}		& \multirow{2}{*}{MEIHT}   	& \multirow{2}{*}{Exact} & \multirow{2}{*}{DP}	& $kN\bar n$ 	& \multirow{2}{*}{$\ell_1/\ell_1$}\\
													&							& 						& 	& $\left(M^2G+N\right)\bar n$						& \\
		\hline
		\multirow{2}{*}{\cite{hegde2015nearly}}	   	& PCSF-TAIL			& \multirow{2}{*}{Approximate} 	& \multirow{2}{*}{--} 	& \multirow{2}{*}{$N\log^4(N)$}		& \multirow{2}{*}{--}\\
		 											& GRAPH-CoSaMP		&  	&  	&		& \\
		\hline
		\cite{kyrillidis2016convex}	   				& $\ell_{21}$-minimization	& Exact 	& {SOCP}	& Polynomial 		& $\ell_{21}/\ell_{21}$\\
		\hline
		   											& MEIHT				& Exact 	& \multirow{2}{*}{DP or BD}	& $(N+M)\left(w_T^2 5^{w_T} G^2 N\right) \bar n$	& \multirow{2}{*}{$\ell_1/\ell_1$}\\
		This										& Model-IHT	    	& Exact 	& 							& or exponential		 						& \\
		work 										& AM-IHT	    	& Approximate & {LP-rounding}	& \multirow{2}{*}{Polynomial}						& \multirow{2}{*}{--}\\
		 											& AM-EIHT	    	& Approximate & {\& greedy}	&  						& \\
		\hline
	\end{tabular}
	\label{tab:model-decoding}
\end{table}

\section{Algorithms with Exact Projection Oracles}\label{sec:exactprojections}
In this section we study the exact \MS~problem which has to be solved iteratively in the Model-IHT and the MEIHT. We extend existing results for group-sparse models and pass from loopless overlapping group models (which was the most general prior to this work) to overlapping group models of bounded treewidth and to general group models without any restriction on the structure. The graph representing a loopless overlapping group model has a treewidth of 1.

We start by showing that it is possible to perform exact projections onto overlapping groups with bounded treewidth using dynamic programming, see Section \ref{sec:treewidth}. While this procedure has a polynomial run-time bound it is restricted to the class of group models with bounded treewidth. Nevertheless we prove that the exact projection problem is \NP-hard if the incidence graph is a grid which is the most basic graph structure without bounded treewidth. For the sake of completeness we solve the exact projection problem for all instances of group models by a method based on Benders' decomposition in Section \ref{sec:exact}. Solving an NP-hard problem this method does not yield a polynomial run-time guarantee but works well in practice as shown in \cite{cordeau2018benders}. In Section \ref{meiht_new} we present an appropriately modified algorithm (MEIHT) with exact projection oracles for the recovery of signals from structured sparsity models. We derive corollaries for the general group-model case from existing works about run-time and convergence of this modified algorithm.  

Recall the following notation: $\cG_i$ denotes a group of size $g_i$, $i\in [M]$, and a group model is a collection $\mathfrak{G} = \{\cG_1,\ldots,\cG_M\}$. The group-sparse model of order $G$ is denoted by $\fG_G$, which contains all supports $\support$ which are contained in the union of at most $G$ groups of $\fG$, i.e. $\support\subseteq \bigcup_{j\in \mathcal{I}} \cG_j, ~ \mathcal{I}\subseteq [M]$ and $|\mathcal{I}|\leq G$; see~\eqref{eqn:G-sparse-model}. We will interchangeably say that $\x$ or $\support$ is $\fG_G$-sparse. Clearly group models are a special case of structured sparsity models. Assume $s_{\text{max}}$ is the size of the maximal support which is possible by selecting $G$ groups out of $\fG$. For $g_{\text{max}}$ denoting the maximal size of a single group in $\mathfrak G$, i.e.,
\[
g_{\text{max}} = \max_{i\in [M]} |\mathcal G_i| ,
\]
we have $s_{\text{max}}\in \mathcal O\left( G g_{\text{max}}\right)$. Furthermore the number of possible supports is in $\mathcal O(M^G)$. Therefore applying the result from Section \ref{sec:PrelimModelBasedCS} we obtain
\begin{equation}\label{eqn:measurementboundgroupmodels}
m=\mathcal O\left( \delta^{-2}G g_{\text{max}} \log(\delta^{-1}) + G\log(M) \right)
\end{equation}
as the number of required measurements for a sub-Gaussian matrix to obtain the group-model-RIP with RIC $\delta$  with high probability, which induces the convergence of Algorithm \ref{alg:Model-IHT} for small enough $\delta$. 

\subsection{Group models of Low Treewidth}\label{sec:treewidth}

One approach to overcome the hardness of the \MS~problem is to restrict the structure of the group models considered.
To this end we follow Baldassarre et al.~\cite{baldassarre2016group} and consider two graphs associated to a group model $\fG$.

The \emph{intersection graph} of $\fG$, $I(\fG)$, is given by the vertex set $V(I(\fG)) = \fG$, and the edge set 
\[E(I(\fG)) = \{ RS : R,S \in \fG, R\neq S \mbox{ and } R \cap S \neq \emptyset \}.\]
The \emph{incidence graph} of $\fG$, $B(\fG)$, is given by the vertex set $V(B(\fG)) = [N] \cup \fG$, and the edge set 
\[E(B(\fG)) = \{ eS : e\in [N], S \in \fG \mbox{ and } e \in S \}.\]
Note that the incidence graph is bipartite since an edge is always adjacent to an element $e$ and a group $S$. See Fig.~\ref{fig:graph-examples} for a simple illustration of the two constructions.
\begin{figure}[h!]
	\centering
		\includegraphics[width=0.55\textwidth]{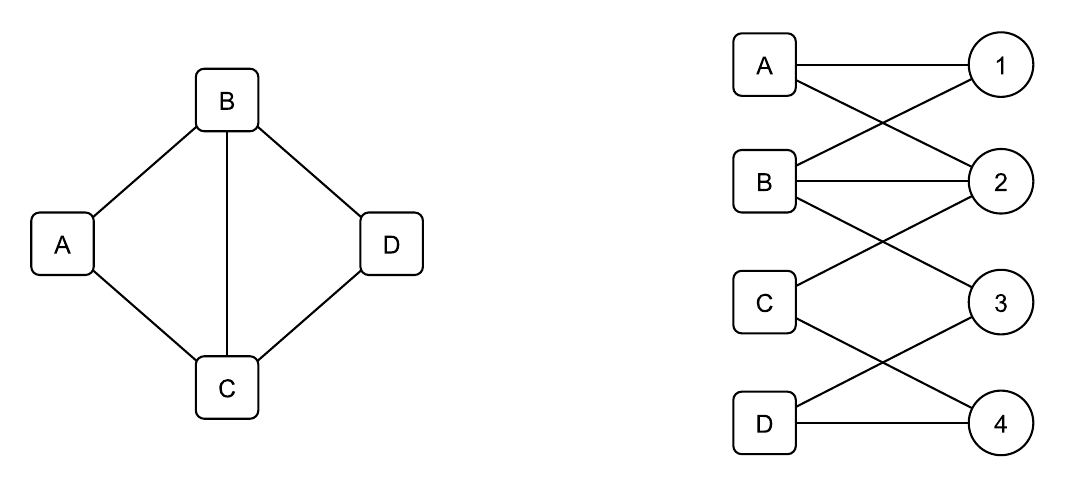}
		\caption{Intersection graph and incidence graph of the group model $\mathfrak G=\{A=\{1,2\},B=\{1,2,3\},C=\{2,4\},D=\{3,4\}\}$.}
	\label{fig:graph-examples}
\end{figure}
Baldassarre et al.~\cite{baldassarre2016group} prove that there is a polynomial time algorithm to solve the \MS~problem in the case that the intersection graph is an acyclic graph.
Their algorithm uses dynamic programming on the acyclic structure of the intersection graph.

We generalize this approach and show that the same problem can be solved in polynomial time if the treewidth of the incidence graph is bounded. Following Proposition~\ref{prop:I(M)-versus-B(M)} below, this implies that the \MS~Problem can be solved in polynomial time if the treewidth of the intersection graph is bounded. We proceed by formally introducing the relevant concepts.
\subsection{Tree Decomposition}
Let $\bar G=(V,E)$ be a graph.
A \emph{tree decomposition} of $\bar G$ is a tree $T$ where each node $x \in V(T)$ of $T$ has a \emph{bag} $B_x \subseteq V$ of vertices of $\bar G$ such that the following properties hold:
\begin{enumerate}
	\item $\bigcup_{x \in V(T)} B_x = V$.
	\item If $B_x$ and $B_y$ both contain a vertex $v \in V$, then the bags of all nodes of $T$ on the path between $x$ and $y$ contain $v$ as well. Equivalently, the tree nodes containing the vertex $v$ form a connected subtree of $T$.
	\item For every edge $vw$ in $E$ there is some bag that contains both $v$ and $w$. That is, vertices in $V$ can be adjacent only if the corresponding subtrees in $T$ have a node in common.
\end{enumerate}
The \emph{width} of a tree decomposition is the size of its largest bag minus one, i.e. $\max_{x \in V(T)} |B_x|-1$. The \emph{treewidth of $\bar G$}, $\mbox{tw}(\bar G)$, is the minimum width among all possible tree decompositions of $\bar G$.

Intuitively, the treewidth measures how `treelike' a graph is: the smaller the treewidth is, the more treelike is the graph.
The graphs of treewidth one are the acyclic graphs. Fig.~\ref{fig:tree-decomposition} shows a graph of treewidth 2, together with a tree decomposition.

\begin{figure}[h]
	\centering
		\includegraphics[width=0.5\textwidth]{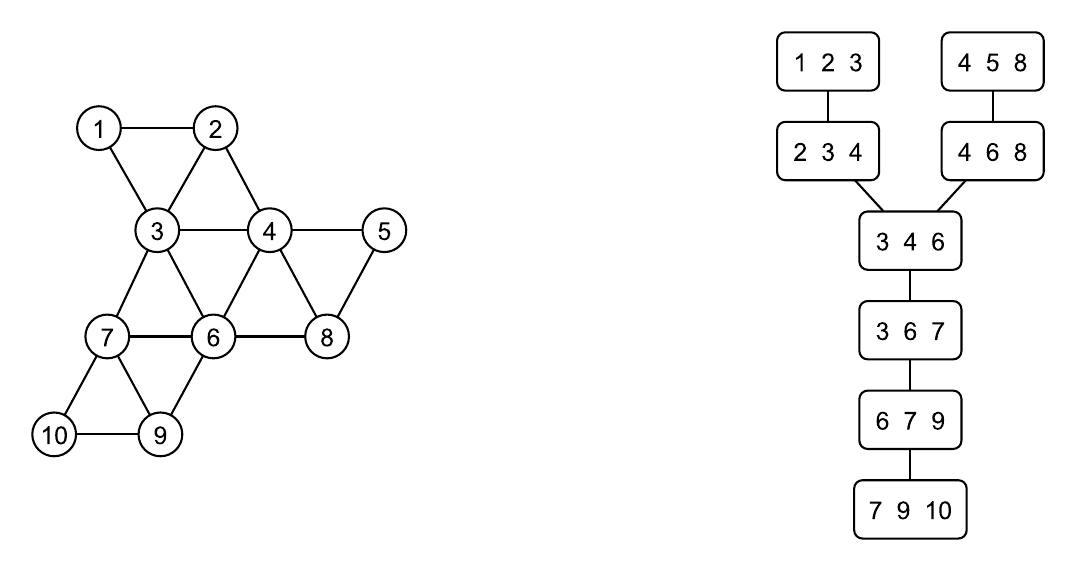}
	\caption{A graph and a tree decomposition of width 2.}
	\label{fig:tree-decomposition}
\end{figure}

Before stating any algorithms, we discuss the relation of the treewidth of the intersection and the incidence graphs of a given group model.

When bounding the treewidth of the graphs associated to a group model, it makes sense to consider the incidence graph rather than the intersection graph.
This is due to the following simple observation.

\begin{proposition}\label{prop:I(M)-versus-B(M)}
For any group model $\fG$ it holds that $\mbox{tw}(B(\fG)) \le \mbox{tw}(I(\fG))+1$.
However, for every $t$ there exists a group model $\fG$ such that $\mbox{tw}(I(\fG))-\mbox{tw}(B(\fG)) = t$.
\end{proposition}

This statement is not necessarily new, but we quickly prove it in our language in order to be self-contained.

\begin{proof}[Proof of Proposition~\ref{prop:I(M)-versus-B(M)}.]
To see the first assertion, let $\fG$ be a group model. Let $T$ be a tree decomposition of $I(\fG)$ of width $\mbox{tw}(I(\fG))$.
In the following, we attach leaves to $T$, one for each element in $[N]$, and obtain a tree decomposition of $B(\fG)$.
Each leaf will contain at most $\mbox{tw}(I(\fG))$ many elements of $\fG$ and at most one element of $[N]$. 
Hence we get $\mbox{tw}(B(\fG)) \le \mbox{tw}(I(\fG))+1$.

To construct the tree decomposition of $B(\fG)$ pick any $i \in [N]$, and let $\fG_i$ be the set of groups in $\fG$ containing $i$.
Since all groups in $\fG_i$ contain $i$, the set $\fG_i$ is a clique in $I(\fG)$.\footnote{Recall that a \emph{clique} in a graph is a set of mutually adjacent vertices.}
Moreover, since $T$ is a tree decomposition of $I(\fG)$, the subtrees of the groups in $\fG_i$ mutually share a node.
As subtrees of a tree have the Helly property, there is at least one node $x$ of $T$ such that $\fG_i \subseteq B_x$.
We now add a new node $x_i$ with bag $B_{x_i} = \fG_i \cup \{i\}$ and an edge between $x_i$ and $x$ in $T$. Doing this for all $i\in [N]$ simultaneously, it is easy to see that we arrive at a tree decomposition $T'$ of $B(\fG)$ of width at most $\mbox{tw}(I(\fG))+1$ which proves the first assertion.

To prove the second assertion consider, for any $t$, the group model $\fG$ where 
$$\fG = \{\cG_i : i\in [t+2]\} \mbox{ with } \cG_i = \{i,t+3\}.$$
Note that $B(\fG)$ is a tree, hence $\mbox{tw}(B(\fG))=1$.
In $I(\fG)$, however, the set $\fG$ is a clique of size $t+2$.
Thus, $\mbox{tw}(I(\fG))=t+1$ which implies $\mbox{tw}(I(\fG))-\mbox{tw}(B(\fG)) = t$. \qed
\end{proof}

Consider now a tree decomposition $T$ of a graph $\bar G$.
We say that $T$ is a \emph{nice tree decomposition} if every node $x$ is of one of the following types.
\begin{itemize}
	\item \textit{Leaf:} $x$ has no children and $B_x=\emptyset$.
	\item \textit{Introduce:} $x$ has one child, say $y$, and there is a vertex $v \notin B_y$ of $\bar G$ with $B_x = B_y \cup \{v\}$.
	\item \textit{Forget:} $x$ has one child, say $y$, and there is a vertex $v \notin B_x$ of $\bar G$ with $B_y = B_x \cup \{v\}$.
	\item \textit{Join:} $x$ has two children $y$ and $z$ such that $B_x=B_y=B_z$.
\end{itemize}
This kind of decomposition limits the structure of the difference of two adjacent nodes in the decomposition. A folklore statement (explained in detail in the classic survey by Kloks~\cite{DBLP:books/sp/Kloks94}) says that such a nice decomposition is easily computed given any tree decomposition of $\bar G$ without increasing the width.

\subsection{Dynamic Programming}
In this section we derive a polynomial time algorithm for the \MS~problem for fixed treewidth. As in Section \ref{sec:Prel_GroupStructures} we assume we have a given signal $x\in\RR^N$ and define $w\in\RR^N$ with $w_i=x_i^2$. In the following we use the notation
\[
w(\mathcal S ):=\sum_{i\in\mathcal G}w_i 
\] 
for a subset $\mathcal S\subseteq [N]$. The algorithm is presented using a nice tree decomposition of the incidence graph $B(\fG)$ and uses the following concept.
Fix a node $x$ of the decomposition tree of $B(\fG)$, a number $i$ with $0 \le i \le G$ and a map $c:B_x\to \{0,1,1_?\}$.
We say that $c$ is a \emph{colouring} of $B_x$.
We consider solutions to the problem \MS$(x,i,c)$ which is defined as follows.

A set $\cS \subseteq \fG$ is a feasible solution of \MS$(x,i,c)$ if
\begin{enumerate}[(a)]
	\item $\cS$ contains only groups that appear in some bag of a node in the subtree rooted at node $x$,
	\item $|\cS| = i$,
	\item $\cS \cap B_x$ contains exactly those group-vertices of $B_x$ that are in $c^{-1}(1)$, that is, 
	\[\cS \cap c^{-1}(1) = \fG \cap c^{-1}(1), \text{ and}\]
	\item of the elements in $B_x$, $\cS$ covers exactly those that are in $c^{-1}(1)$. Formally, \[\left(\bigcup \cS\right) \cap B_x = [N]  \cap c^{-1}(1).\]
\end{enumerate}
The objective value of the set $\cS$ is given by $w(\bigcup \cS) + w(c^{-1}(1_?))$.
Intuitively, a feasible solution to \MS$(x,i,c)$ does not cover elements labelled $0$ or $1_?$, but covers all elements labelled $1$.
The elements labelled $1_?$ are assumed to be covered by groups not yet visited in the tree decomposition.

If \MS$(x,i,c)$ does not admit a feasible solution, we say that \MS$(x,i,c)$ is infeasible.
The maximum objective value attained by a feasible solution to \MS$(x,i,c)$, if feasible, we denote by $\mbox{OPT}(x,i,c)$.
If \MS$(x,i,c)$ is infeasible, we set $\mbox{OPT}(x,i,c) = - \infty$.

Assertion (d) implies that \MS$(x,i,c)$ is infeasible if the groups in $c^{-1}(1)$ cover elements in $c^{-1}(0)$ or $c^{-1}(1_?)$. 
That is, \MS$(x,i,c)$ is infeasible if $\bigcup \left(\mathfrak G \cap c^{-1}(1)\right) \not\subseteq [N] \cap c^{-1}(1)$.
To deal with this exceptional case we call $c$ \emph{consistent} if $\bigcup \left(\mathfrak G \cap c^{-1}(1)\right) \subseteq [N] \cap c^{-1}(1)$, and \emph{inconsistent} otherwise.
Note that consistency of $c$ is necessary to ensure feasibility of \MS$(x,i,c)$, but not sufficient.

Our algorithm processes the nodes of a nice tree decomposition in a bottom-up fashion.
Fix a node $x$, a number $i$ with $0 \le i \le G$ and a map $c:B_x\to \{0,1,1_?\}$.
We use dynamic programming to compute the value $\mbox{OPT}(x,i,c)$, assuming we know all possible values $\mbox{OPT}(y,j,c')$ for all children $y$ of $x$, all $j$ with $0 \le j \le G$, and all $c': B_y\to \{0,1,1_?\}$. 

In the following, for a subset $S\subset B_x$ the function $c|_S:S\to \{0,1,1_?\}$ is the restriction of $c$ to $S$. 
We use $\Gamma(v)$ to denote the neighborhood of $v$ in $B(\fG)$. 

If $c$ is not consistent, we may set $\mbox{OPT}(x,i,c) = - \infty$ right away.
We thus assume that $c$ is consistent and distinguish the type of node $x$ as follows.
\begin{itemize}
	\item \textit{Leaf:} set $\mbox{OPT}(x,0,c) = 0$ and  $\mbox{OPT}(x,i,c) = -\infty$ for all $i \in [G]$.
	\item \textit{Introduce:} let $y$ be the unique child of $x$ and let $v \notin B_y$ such that $B_x = B_y \cup \{v\}$.

If $v \in [N]$, we set
\begin{equation*}
\mbox{OPT}(x,i,c) =
\begin{cases}
\mbox{OPT}(y,i,c|_{B_y}), \text{ if } c(v)=0\\
\mbox{OPT}(y,i,c|_{B_y}) + w(v), \text{ if } c(v)=1 \text{ and } i > 0\\
\mbox{OPT}(y,i,c|_{B_y}) + w(v), \text{ if } c(v)=1_?\\
-\infty, \text{ otherwise}
\end{cases}
\end{equation*}

If $v \in \fG$, we set
\begin{equation}\label{eqn:join-eqn}
\mbox{OPT}(x,i,c) =
\begin{cases}
\mbox{OPT}(y,i,c|_{B_y}), \text{ if } c(v)=0\\
\max\{ \mbox{OPT}(y,i-1,c') : (y,c') \text{ is compatible to } (x,c) \}, \text{ if } c(v)=1 \text{ and } c^{-1}(1) \cap \Gamma(v) \neq \emptyset\\
-\infty, \text{ otherwise}
\end{cases}
\end{equation}
where $(y,c')$ is \emph{compatible} to $(x,c)$ if
\begin{itemize}
	\item $c':B_y \to \{0,1,1_?\}$ is a consistent colouring of $B_y$,
	\item $c^{-1}(0) = c'^{-1}(0)$, and
	\item $c^{-1}(1) = c'^{-1}(1) \cup (c'^{-1}(1_?) \cap \Gamma(v))$.
\end{itemize}
	\item \textit{Forget:} let $y$ be the unique child of $x$ and let $v \notin B_x$ such that $B_y = B_x \cup \{v\}$. We set
\begin{equation}\label{eqn:forget}
\mbox{OPT}(x,i,c) = \max\{\mbox{OPT}(y,i,c') : ~ c':B_y \to \{0,1,1_?\} \mbox{, } c = c'|_{B_x} \text{, and } c'(v) \neq 1_?\}
\end{equation}
	\item \textit{Join:} we set
\begin{multline}
\mbox{OPT}(x,i,c) = \max\{\mbox{OPT}(y,i_1,c') + \mbox{OPT}(z,i_2,c'') \\
- w(((B_x \cap [N]) \cup \bigcup (B_x \cap \fG)) \setminus c^{-1}(0)) : i_1 + i_2 - |c'^{-1}(1) \cap c''^{-1}(1) \cap \fG| = i\}, \label{eqn:join-node}
\end{multline}
where $y$ and $z$ are the two children of $x$. The maximum is taken over all consistent colourings $c',c'':B_x \to \{0,1,1_?\}$ with $c^{-1}(0)=c'^{-1}(0)=c''^{-1}(0)$ and $c^{-1}(1) = c'^{-1}(1) \cup c''^{-1}(1)$.
	\item \textit{Root:} first we compute $\mbox{OPT}(x,G,c)$ for all relevant choices of $c$, depending on the type of node $x$. The algorithm then terminates with the output 
	\[
	\mbox{OPT} = \max \left\{ \mbox{OPT}(x,G,c) : ~ c : B_x \to \{0,1,1_?\}\right\}.
	\]
\end{itemize}

\begin{lemma}\label{lem:correctness-dyn-prog}
The output OPT is the objective value of an optimal solution of the \MS~problem.
\end{lemma}
\begin{proof}
The proof follows the individual steps of the dynamic programming algorithm.
Consider the problem \MS$(x,i,c)$.
If $c$ is not consistent, we correctly set $\mbox{OPT}(x,i,c)=-\infty$.
We thus proceed to the case when $c$ is consistent.
Fix an optimal solution $\cS$ of \MS$(x,i,c)$ if existent.

The \textit{leaf}-node case is clear, so we proceed to the case of $x$ being an \textit{introduce}-node.
Let $y$ be the unique child of $x$ and let $v \notin B_y$ such that $B_x = B_y \cup \{v\}$.
First assume that $v \in [N]$.

Assume that $c(v)=0$.
Since $c$ is consistent, $c^{-1}(1) \cap \Gamma(v) = \emptyset$ holds, and so $\cS$ is an optimal solution to \MS$(y,i,c|_{B_y})$.
We may thus set $\mbox{OPT}(x,i,c)=\mbox{OPT}(y,i,c|_{B_y})$.

If $c(v)=1$, $\cS$ covers $v$, so we need to make sure some vertex $u \in \Gamma(v) \cap B_x$ is contained in the solution in order for \MS$(x,i,c)$ to be feasible.
Hence we have $\mbox{OPT}(x,i,c) = \mbox{OPT}(y,i,c|_{B_y}) + w(v)$ if $c^{-1}(1) \cap \Gamma(v) \neq \emptyset$, and $\mbox{OPT}(x,i,c)=-\infty$ otherwise.

Next we assume that $v \in \fG$.
If $c(v)=0$, we may simply put $\mbox{OPT}(x,i,c) = \mbox{OPT}(y,i,c|_{B_y})$.
So, assume that $c(v)=1$.
If \MS$(x,i,c)$ is feasible and thus $\cS$ exists, define $\cS' = \cS \setminus \{v\}$.
Now $\cS'$ is a solution to $\mbox{OPT}(y,i,c')$ for some $c' : B_y \to \{0,1,1_?\}$ with $c^{-1}(0) = c'^{-1}(0)$ and $c^{-1}(1) = c'^{-1}(1) \cup (c'^{-1}(1_?) \cap \Gamma(v))$.
Note that $(y,c')$ is compatible to $(x,c)$.
Consequently, $\mbox{OPT}(x,i,c)$ is upper bounded by the right hand side of \eqref{eqn:join-eqn}.

To see that $\mbox{OPT}(x,i,c)$ is at least the right hand side of \eqref{eqn:join-eqn}, let $(y,c'')$ be compatible to $(x,c)$ and let $\cS''$ be a solution to \MS$(y,i,c'')$ of objective value $\lambda\in \RR$.
Then $\cS'' \cup \{v\}$ is a solution to \MS$(x,i,c)$ of objective value $\lambda$.
Consequently, 
\[\mbox{OPT}(x,i,c) = \max \{ \mbox{OPT}(y,i,c') : (y,c') \text{ is compatible to } (x,c) \}.\] 

If $x$ is a \textit{forget}-node, let $y$ be the unique child of $x$ and let $v \notin B_x$ such that $B_y = B_x \cup \{v\}$.
If $v \in \cS$, we have
\[
\mbox{OPT}(x,i,c) \le \mbox{OPT}(y,i,c') \mbox{ where } c':B_y \to \{0,1,1_?\}, \ \ c = c'|_{B_x} \text{ and } c'(v) =1.
\]
Otherwise, if $v \notin \cS$, we have
\[
\mbox{OPT}(x,i,c) \le \mbox{OPT}(y,i,c') \mbox{ where } c':B_y \to \{0,1,1_?\}, \ \ c = c'|_{B_x} \text{ and } c'(v) =0.
\]
Moreover, any solution of \MS$(y,i,c')$, where $c':B_y \to \{0,1,1_?\}, \ \ c = c'|_{B_x}$ and $c'(v) =1$ is a solution of \MS$(x,i,c)$.
This proves \eqref{eqn:forget}.

If $x$ is a \textit{join}-node, let $y$ and $z$ be the two children of $x$ and recall that $B_x=B_y=B_z$.

Let $\cS'$ be the collection of groups in $\cS$ contained in the subtree rooted at $y$, and let $\cS''$ be the collection of groups in $\cS$ contained in the subtree rooted at $z$.
Since $T$ is a tree decomposition, $\cS' \cap \cS'' = \cS \cap B_x$.

Note that $\cS'$ is a solution of \MS$(y,i_1,c')$ and $\cS''$ is a solution of \MS$(z,i_1,c'')$ for some $c',c'':B_x \to \{0,1,1_?\}$ and $i_1,i_2$ with $i_1 + i_2 - |c'^{-1}(1) \cap c''^{-1}(1) \cap \fG| = i$. 
It is easy to see that $c^{-1}(0)=c'^{-1}(0)=c''^{-1}(0)$ and $c^{-1}(1) = c'^{-1}(1) \cup c''^{-1}(1)$.

The objective value of $\cS$ equals 
\[w(\bigcup \cS) + w(c^{-1}(1_?)) = w(\bigcup \cS') + c'^{-1}(1_?) + w(\bigcup \cS'') + c''^{-1}(1_?) - w(((B_x \cap [N]) \cup \bigcup (B_x \cap \fG)) \setminus c^{-1}(0)).\]
This shows that $\mbox{OPT}(x,i,c)$ is at most the right hand side of \eqref{eqn:join-node}.

Now let $\tilde \cS$ be an optimal solution of \MS$(y,j_1,\tilde c)$ and let $\hat \cS$ be an optimal solution of \MS$(z,j_2,\hat c)$ where 
\begin{itemize}
	\item $\tilde c, \hat c:B_x \to \{0,1,1_?\}$ are both consistent,
	\item $c^{-1}(0)=\tilde c^{-1}(0)=\hat c^{-1}(0)$ and $c^{-1}(1) = \tilde c^{-1}(1) \cup \hat c^{-1}(1)$, and
	\item $j_1 + j_2 - |\tilde c^{-1}(1) \cap \hat c^{-1}(1) \cap \fG| = i$.
\end{itemize}
Note that $\tilde \cS$ and $\hat \cS$ exist since, as we have shown earlier, the colourings $c'$ and $c''$ satisfy the above assertions.
Then $\hat \cS \cup \tilde \cS$ is a solution of \MS$(x,i,c)$ with objective value
$$\max \{\mbox{OPT}(y,j_1,\tilde c) + \mbox{OPT}(z,j_1,\hat c) - w(((B_x \cap [N]) \cup \bigcup (B_x \cap \fG)) \setminus c^{-1}(0))\}.$$
Consequently, $\mbox{OPT}(x,i,c)$ is at least the right hand side of \eqref{eqn:join-node} and thus \eqref{eqn:join-node} holds. \qed
\end{proof}

By storing the best current solution alongside the OPT$(x,i,c)$-values we can compute an optimal solution together with OPT.

\subsubsection{Runtime of the Algorithm}
The computational complexity of the individual steps are as follows.
\begin{enumerate}[(a)]
	\item Given the incidence graph $B(\fG)$ on $n=M+N$ vertices of treewidth $w_T$, one can compute a tree decomposition of width $w_T$ in time $\bigO(2^{\bigO(w_T^3)}n)$ using Bodlaender's algorithm~\cite{DBLP:journals/siamcomp/Bodlaender96}. The number of nodes of the decomposition is in $\bigO(n)$.
	\item Given a tree decomposition of width $w_T$ with $t$ nodes, one can compute a nice tree decomposition of width $w_T$ on $\bigO(w_Tt)$ nodes in $\bigO(w_T^2t)$ time in a straightforward way~\cite{DBLP:books/sp/Kloks94}. 
\end{enumerate}

The running time of the dynamic programming is bounded as follows.

\begin{theorem}\label{thm:runtime-dynprog}
The dynamic programming algorithm can be implemented to run in $\bigO(w_T 5^{w_T} G^2 N t)$ time given a nice tree decomposition of $B(\fG)$ of width $w_T$ on $t$ nodes.
\end{theorem}

Note that we can assume that $t = \bigO(w_T n)$ with $n=M+N$. Together with the running time of the construction of the nice tree decomposition we can solve the exact projection problem on graphs with treewidth $w_T$ in $\bigO((N+M)(w_T^2 5^{w_T} G^2 N + 2^{\bigO(w_T^3)}+w_T^2))$.

\begin{proof}[Proof of Theorem~\ref{thm:runtime-dynprog}]
Since the \textit{join}-nodes are clearly the bottleneck of the algorithm, we discuss how to organize the computation in a way that the desired running time bound of $\bigO(w_T 5^{w_T} G^2 N)$ holds in a node of this type.

So, let $x$ be a \textit{join}-node and assume that $y$ and $z$ are the two children of $x$.
We want to compute $\mbox{OPT}(x,i,c)$, for all colourings $c:B_x \to \{0,1,1_?\}$ and all $i$ with $0 \le i \le G$.
Recall that we need to compute this value according to \eqref{eqn:join-node}, that is,
\begin{align*}
\mbox{OPT}(x,i,c) = \max\{&\mbox{OPT}(y,i_1,c') + \mbox{OPT}(z,i_2,c'') \nonumber \\
&- w(((B_x \cap [N]) \cup \bigcup (B_x \cap \fG)) \setminus c^{-1}(0)) : i_1 + i_2 - |c'^{-1}(1) \cap c''^{-1}(1) \cap \fG| = i\},
\end{align*}
where the maximum is taken over all consistent colourings $c',c'':B_x \to \{0,1,1_?\}$ with $c^{-1}(0)=c'^{-1}(0)=c''^{-1}(0)$ and $c^{-1}(1) = c'^{-1}(1) \cup c''^{-1}(1)$.

We enumerate all $5^{w_T+1}$ colourings $C: B_x \to \{(0,0),(1,1),(1,1_?),(1_?,1),(1_?,1_?)\}$ and derive $c$, $c'$, and $c''$. 
We put
\begin{align*}
c(v) &=
\begin{cases}
0, \text{ if } C(v)=(0,0)\\
1, \text{ if } C(v) \in \{(1,1),(1,1_?),(1_?,1)\}\\
1_?, \text{ if } C(v)=(1_?,1_?)\\
\end{cases}\\
c'(v) &=
\begin{cases}
0, \text{ if } C(v)=(0,0)\\
1, \text{ if } C(v) \in \{(1,1),(1,1_?)\}\\
1_?, \text{ if } C(v) \in \{(1_?,1),(1_?,1_?)\}\\
\end{cases}\\
c''(v) &=
\begin{cases}
0, \text{ if } C(v)=(0,0)\\
1, \text{ if } C(v) \in \{(1,1),(1,1_?)\}\\
1_?, \text{ if } C(v) \in \{(1,1_?),(1_?,1_?)\}\\
\end{cases}
\end{align*}
If either of $c$, $c'$, or $c''$ are inconsistent, we discard this choice of $C$.
In this way we capture every consistent colouring $c:B_x \to \{0,1,1_?\}$, and all consistent choices of $c'$, and $c''$ satisfying $c^{-1}(0)=c'^{-1}(0)=c''^{-1}(0)$ and $c^{-1}(1) = c'^{-1}(1) \cup c''^{-1}(1)$.

It remains to discuss the computation of the value $w((c^{-1}(1) \cap [N]) \cup \bigcup (c^{-1}(1) \cap \fG)))$.
This value can be computed in $\bigO(w_TN)$ time, since we are computing differences and unions of at most $w_T$ groups of size $N$ each.
We arrive at a total running time in $\bigO(w_T 5^{w_T} G^2 N)$.
\qed
\end{proof}
\begin{remark}\label{rem:kplusG}
	The dynamic programming algorithm can be extended to include a sparsity restriction on the support of the signal approximation itself. So, we can compute an optimal $K$-sparse $G$-group-sparse signal approximation if the bipartite incidence graph of the studied group models is bounded. The running time of the algorithm increases by a factor of $\bigO(K)$.
\end{remark}

\subsection{Hardness on Grid-Like Group Structures}\label{hardness}

An $r \times r$\emph{-grid} is a graph with vertex set $[r]\times[r]$, and two vertices $(a,b),(c,d) \in [r]\times[r]$ are adjacent if and only if $|a-c|=1$ and $|b-d|=0$, or if $|a-c|=0$ and $|b-d|=1$. We also say that $r$ is the size of the grid. Fig.~\ref{fig:grid} shows a $6 \times 6$-grid.

\begin{figure}[h!]
	\centering
		\includegraphics[width=0.2\textwidth]{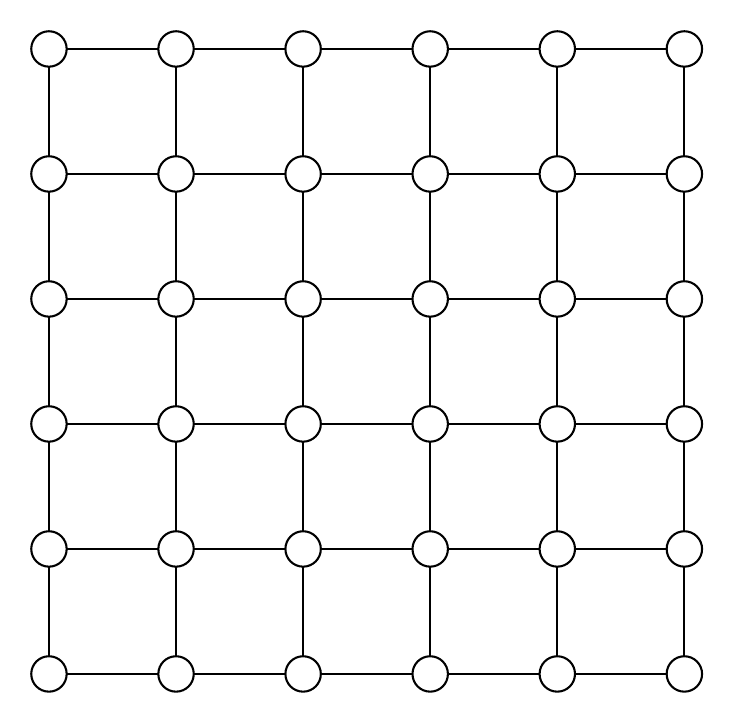}
		\caption{A $6 \times 6$-grid}
	\label{fig:grid}
\end{figure}

Recall the \MS~problem can be solved efficiently when the treewidth of the incidence graph of the group structure is bounded, as shown in Section~\ref{sec:treewidth}.

\begin{definition}[Graph minor]
Let $G_1$ and $G_2$ be two graphs.
The graph $G_2$ is called a \emph{minor} of $G_1$ if $G_2$ can be obtained from $G_1$ by deleting edges and/or vertices and by contracting edges. 
\end{definition}

A classical theorem by Robertson and Seymour~\cite{DBLP:journals/jct/RobertsonS86} says that in a graph class $\mathcal C$ closed under taking minors either the treewidth is bounded or $\mathcal C$ contains all grid graphs.

Consequently, if $\mathcal C$ is a class of graphs that does not have a bounded treewidth, it contains all grids.
Our next theorem shows that \MS~is NP-hard on group models $\fG$ for which $B(\fG)$ is a grid, thus complementing Theorem~\ref{thm:runtime-dynprog}.

\begin{theorem}\label{thm:hardness-on-grids}
The \MS~problem is NP-hard even if restricted to instances $\fG$ for which $B(\fG)$ is a grid graph and the weight of any element is either 0 or 1.
\end{theorem}

Consider the following problem: given an $n \times n$-pixel black-and-white image, pick $k$ $2 \times 2$-pixel windows to cover as many black pixels as possible.
This problem can be modeled as the \MS~problem on a grid graph where the weight of any element is either 0 or 1. 
See Fig.~\ref{fig:pixels_to_grid} for an illustration.
Note that this group model is of frequency at most 4, and so we can do an approximate model projection and signal recovery using the result of Section~\ref{sec:low-freq}.

\begin{figure}[h!]
	\centering
		\includegraphics[width=0.5\textwidth]{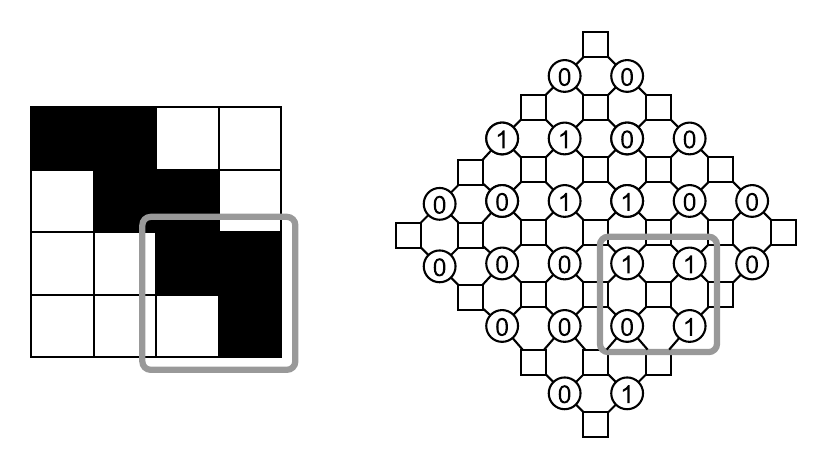}
		\caption{Selecting black pixels by 2x2 pixel windows / \MS~in a grid graph with binary weights.
		Squares are group-vertices, circles are element-vertices}
	\label{fig:pixels_to_grid}
\end{figure}

Our proof is a reduction from the \textsc{Vertex Cover} problem.
Recall that for a graph $\bar G$, a \emph{vertex cover} is a subset of the vertices of $\bar G$ such that any edge of $\bar G$ has at least one endpoint in this subset.
The size of the smallest vertex cover of $\bar G$, the \emph{vertex cover number}, is denoted $\tau(\bar G)$.

Given a graph $\bar G$ and a number $k$ as input, the task in the \textsc{Vertex Cover} problem is to decide whether $\bar G$ admits a vertex cover of size $k$. That is, whether $\tau(\bar G) \le k$.
This problem is NP-complete even if restricted to cubic planar graphs~\cite{DBLP:books/fm/GareyJ79}.\footnote{Recall that a graph is \emph{cubic} if every vertex is of degree $3$, and \emph{planar} if it can be drawn into the plane such that no two edges cross.}

We use the following simple lemma in our proof.

\begin{lemma}\label{lem:vc-contraction}
Let $\bar G$ be a graph and let $G'$ be the graph obtained by subdividing some edge of $\bar G$ twice.
Then $\tau(\bar G) = \tau(G')-1$.
\end{lemma}

We can now prove our theorem.

\begin{proof}[Proof of Theorem~\ref{thm:hardness-on-grids}]
The reduction is from \textsc{Vertex Cover} on planar cubic graphs.
Consider $\bar G=(V,E)$ to be a planar cubic graph and let $k$ be some number.
Our aim is to compute an instance $(\fG,\w,k')$ of the \MS~problem where $B(\fG)$ is a grid such that $\bar G$ has a vertex cover of size $k$ if and only if $\fG$ admits a selection of $k'$ groups that together cover elements of a total weight at least some threshold $t$.

First we embed the graph $\bar G$ in some grid $H$ of polynomial size, meaning the vertices of $\bar G$ get mapped to the vertices of the grid and edges get mapped to mutually internally disjoint paths in the grid connecting its endvertices.
This can be done in polynomial time using an algorithm for orthogonal planar embedding~\cite{DBLP:conf/latin/AlamBKKKW14}.
We denote the mapping by $\pi$, hence $\pi(u)$ is some vertex of $H$ and $\pi(vw)$ is a path from $\pi(v)$ to $\pi(w)$ in $H$ for all $u \in V$ and $vw \in E$.

Next we subdivide each edge of the grid 9 times, so that a vertical/horizontal edge of $H$ becomes a vertical/horizontal path of length 10 in some new, larger grid $H'$.
We choose $H'$ such that the corners of $H$ are mapped to the corners of $H'$.
In particular, $|V(H')| \le 100 |V(H)|$.
Let us denote the obtained embedded subdivision of $\bar G$ in $H'$ by $G'$, and let $\pi'$ denote the embedding.
Moreover, let $\phi$ be the corresponding embedding of $H$ into the subdivided grid $H'$.
Note that $\mbox{im}~\pi'|_{V} \subseteq \mbox{im}~\phi|_{V}$.

Let $(A,B)$ be a bipartition of $H'$.
We may assume that $\pi'(u)$ is in $A$ for all $u \in V$.
We consider $H'$ to be the incidence graph $B(\fG)$ of a group model $\fG$ where the vertices in $B$ correspond to the elements and the vertices in $A$ correspond to the groups of $\fG$.
We refer to the vertices in $A$ as \emph{group-vertices} and to vertices in $B$ as \emph{element-vertices}.
Slightly abusing notation, we identify each group with its group-vertex and each element with its element-vertex and write $\fG=A$.

We observe that
\begin{enumerate}[(a)]
	\item $G'$ is an induced subgraph of $H'$,
	\item every vertex $\pi'(u)$, $u \in V$, has degree 3 in $G'$ and is a group-vertex,
	\item every other vertex has degree 2 in $G'$, and
	\item for every group-vertex $x \in V(H') \setminus V(G')$ there is some group-vertex $u \in V(G')$ with 
	$$\Gamma_{H'}(x) \cap V(G') \subseteq \Gamma_{H'}(u) \cap V(G').$$
\end{enumerate}

Next we will tweak the embedding of $\bar G$ a bit, to get rid of paths $\pi(uv)$ with the wrong parity.
We do so in a way that preserves the properties (a)-(d).
Let $\mathcal P_0 \subseteq \{\pi'(uv) : uv \in E(H)\}$ be the set of all paths with a length 0 (mod 4), and let $\mathcal P_2 = \{\pi'(uv) : uv \in E(H)\} \setminus \mathcal P_0$.
We want to substitute each path in $\mathcal P_0$ by a path of length 2 (mod 4).
For this, let $u'$ be the neighbour of $u$ in the path $\pi(uv)$.
Note that the path $\pi'(uu')$ in $H'$ starts with a vertical or horizontal path $P$ from $\pi'(u)$ to $\pi'(u')$ of length 10.
We bypass the middle vertex of this path (an element-vertex) by going over two new element-vertices and one group-vertex instead. See Fig.~\ref{fig:grid-bypass} for an illustration.

To keep the notation easy we denote the newly obtained path by $\pi''(uv)$.
Note that, after adding the bypass, the new path $\pi''(uv)$ is two edges longer and thus has length 2 (mod 4).
We complete $\pi''$ to an embedding of $\bar G$ by putting $\pi''(u) = \pi'(u)$ and $\pi''(vu') = \pi'(vu')$ for all $u \in V$ and $vu' \in E$ with $\pi'(vu') \in \mathcal P_2$.
Moreover, let us denote the changed embedding of $\bar G$ by $G''$.

\begin{figure}[h!]
	\centering
		\includegraphics[width=0.4\textwidth]{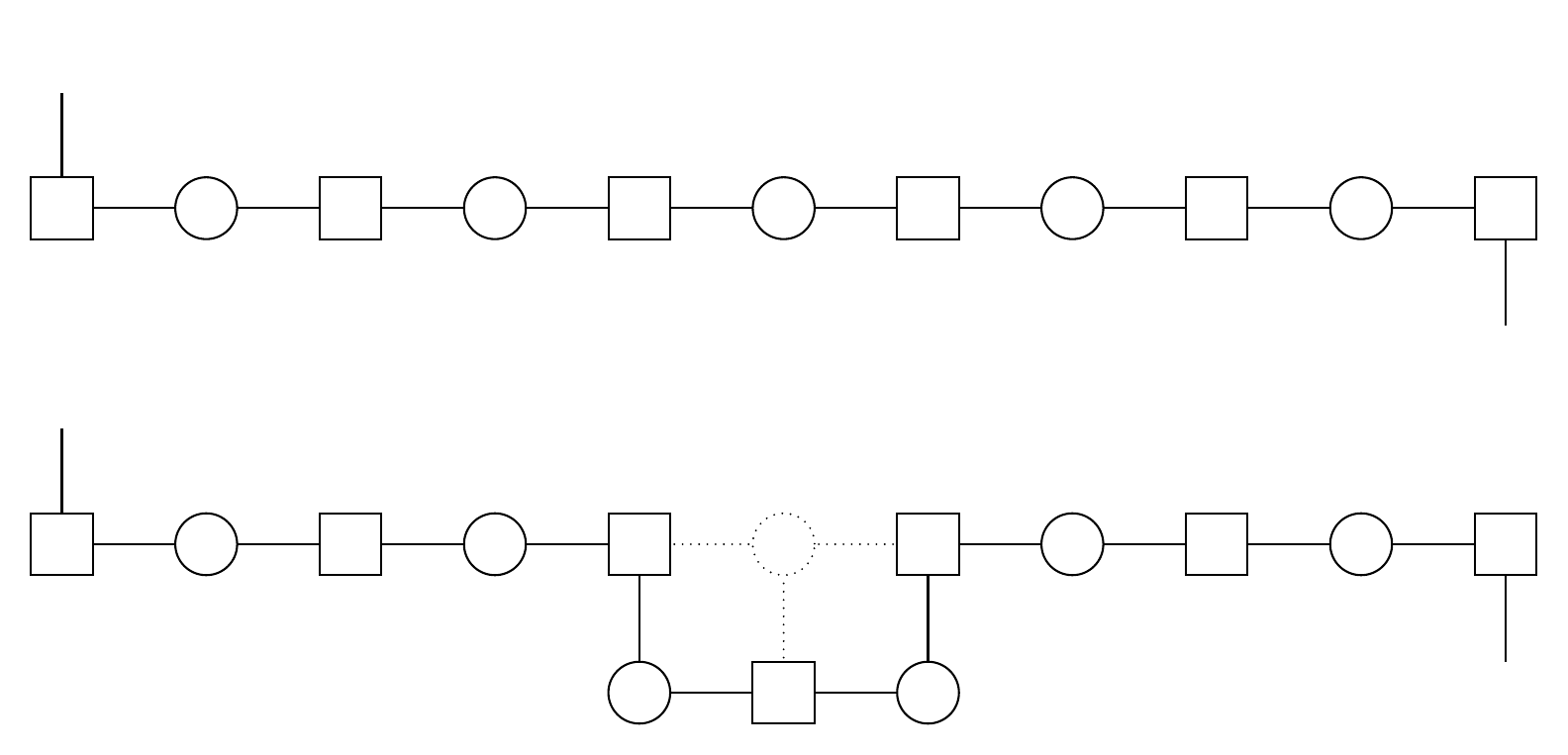}
		\caption{Introducing a bypass. Squares are group-vertices, circles are element-vertices}
	\label{fig:grid-bypass}
\end{figure}

We observe that the new embedding $G''$ still satisfies the assertions (a)-(d) and, in addition, it holds that
\begin{enumerate}[(e)]
	\item every path connecting two vertices of degree 3 over vertices of degree 2 only has length 2 (mod 4).
\end{enumerate}
Next we define the weights of the element-vertices by putting
\[
\w(x) = 
\begin{cases}
1, & x \in V(G'') \\
0, & x \in V(H') \setminus V(G'')
\end{cases} ~~~ \mbox{ for each element-vertex $x$ of $H'$.}
\]
Assertion (d) implies that, for any subset $\mathcal S \subseteq \fG$ of size $k$ there is an $\mathcal S' \subseteq \fG$ of size at most $k$ such that 
\begin{itemize}
	\item $\cS' \subseteq A \cap V(G'')$, and
	\item $\w(\bigcup \cS') \ge \w(\bigcup \cS)$.
\end{itemize}
Since $\w(u)=0$ for all elements in $B \setminus V(G'')$, we may thus restrict our attention to the restricted group model $\fG'= A \cap V(G'')$ on the element set $B\cap V(G'')$.

Slightly abusing notation, any subset $\cS \subseteq \fG'$ is a vertex subset of $I(\fG')$ and $\w(\bigcup \cS)$ equals the number of edges of $I(\fG')$ adjacent to the vertex set $\mathcal S$ in $I(\fG')$.
Moreover, the graph $I(\fG')$ is obtained from the graph $\bar G$ by subdividing each edge an even number of times.

From Lemma~\ref{lem:vc-contraction} we know that there is some number $t$ such that $\tau(\bar G) = \tau(I(\fG))-t$.
Hence, $\bar G$ has a vertex cover of size $k$ if and only if $M'$ has a cover of size $k'=k+t$ of total weight $|E(I(\fG'))|$.
This, in turn, is the case if and only if $M$ admits a cover of size $k'$ of total weight $|E(I(\fG'))|$. 
Since the construction of $\fG$ can be done in polynomial time, the proof is complete. \qed
\end{proof}

\subsection{MEIHT for General Group Structures} \label{meiht_new}
In this section we apply the results for structured sparsity models and for expander matrices to the group model case. The Model-Expander IHT (MEIHT) algorithm is one of the exact projection algorithms with provable guarantees for tree-sparse and loopless overlapping group-sparse models using model-expander sensing matrices \cite{bah2014model}. In this work we show how to use the MEIHT for more general group structures. The only modification of the MEIHT algorithm is the projection on these new group structures. We show MEIHT's guaranteed convergence and polynomial runtime.

\begin{algorithm}\caption{(MEIHT)}\label{alg:meiht}
	\inp $\A$, $\obs$, $\noise$ with $\A\x=\obs+\noise$, $G$, and $\fG$\\
	\out $\text{MEIHT}(\obs,\A,\fG,G)$, a $\fG_G$-sparse approximation to $\x$
	\begin{algorithmic}[1]
		\State $\x^{(0)} \gets \boldsymbol{0}$; \ $n\gets 0$
		\While {halting criterion is false}   
		\State $\x^{(n+1)} \gets \mathcal{P}_{\fG_G}\left[\x^{(n)} + \med(\obs-\A\x^{(n)})\right]$
		\State $n\gets n+1$
		\EndWhile
		\State Return: $\text{MEIHT}(\obs,\A,\fG,G) \gets \x^{(n)}$
	\end{algorithmic}
\end{algorithm}

Note that as in \cite{baldassarre2016group}, we are able to do model projections with an additional sparsity constraint, i.e. projection onto $\mathcal{P}_{\fG_{G,K}}$ defined in \eqref{eqn:GK-sparse-model}. Therefore Algorithm \ref{alg:meiht} works with an extra input $K$ and the model projection $\mathcal{P}_{\fG_{G}}$ becomes $\mathcal{P}_{\fG_{G,K}}$, retuning a $\fG_{G,K}$-sparse approximation to $\x$.

The convergence analysis of MEIHT with the more general group structures considered here remains the same as for loopless overlapping group models discussed in \cite{bah2014model}. We are able to perform the exact projection of $\mathcal{P}_{\fG_G}$ (and  $\mathcal{P}_{\fG_{G,K}}$) as discussed in Section \ref{sec:exact}. With the possibility of doing the projection onto the model, we present the convergence results in Corollaries \ref{cor:convergence0} and \ref{cor:convergence1} as corollaries to Theorem 3.1 and Corollary 3.1 in \cite{bah2014model} respectively.

\begin{corollary}\label{cor:convergence0}
	Consider $\fG$ to be a group model of bounded treewidth and $\support$ to be $\fG_G$-sparse. Let the matrix $\A\in \{0,1\}^{m\times N}$ be a model expander matrix with $\e_{\fG_{3G}} < 1/12$ and $d$ ones per column. For any $\x\in \RR^N$ and $\noise\in \RR^m$, the sequence of updates $\x^{(n)}$ of MEIHT with $\obs = \A\x + \noise$ satisfies, for any $n\geq 0$ 
	\begin{equation}
	\label{eqn:MEIHT_error}
		\|\x^{(n)} - \x_{\support}\|_1 \leq \alpha^n \|\x^{(0)} - \x_{\support}\|_1 + \left(1-\alpha^n\right)\beta\|\A\x_{\support^c} + \noise\|_1,
	\end{equation}
	where $\alpha = 8\e_{\fG_{3G}}\left(1-4\e_{\fG_{3G}}\right)^{-1} \in (0,1)$ and $\beta = 4d^{-1}\left(1 - 12\e_{\fG_{3G}}\right)^{-1} \in (0,1)$.
\end{corollary}
Note that $\e_{\fG_{3G}}$ is the expansion coefficient of the underlying $(s,d,\e_{\fG_{3G}})$-model expander graph for $\A$. This ensures that $\A$ satisfied model RIP-1 over all $\fG_{3G}$-sparse signals.

The proof of this Corollary can be done analogously to the proof of Theorem 3.1 in \cite{bah2014model}. It is thus skipped and the interested reader is referred to \cite{bah2014model}.
 
Let us define the $\ell_1$-error of the best $\fG_G$-term approximation to a vector $\x\in\RR^N$ 
\begin{equation}
	\sigma_{\fG_G}(\x)_1 = \min_{{\fG_G}\mbox{-sparse} ~\z}\|\x-\z\|_1.
\end{equation}
This is then used in the following corollary.

\begin{corollary}\label{cor:convergence1}
	Consider the setting of Corollary \ref{cor:convergence0}. After $n = \left\lceil\log\left(\frac{\|\x\|_1}{\|\noise\|_1}\right) / \log \left(\frac{1}{\alpha}\right) \right\rceil$ iterations, MEIHT returns a solution $\sol$ satisfying
	\begin{equation}
	\label{eqn:lim_MEIHT_error}
		\|\sol - \x\|_1 \leq c_1 \sigma_{\fG_G}(\x)_1 + c_2 \|\noise\|_1
	\end{equation}
	where $c_1 = \beta d$ and $c_2 = 1+\beta$. 
\end{corollary}

\begin{proof}


Without loss of generality we initialize MEIHT with $\x^{(0)} = 0$. Upper bounding $1-\alpha^n$ by $1$ and using triangle inequalities with some algebraic manipulations, \eqref{eqn:MEIHT_error} simplifies to
\begin{equation}
	\label{eqn:MEIHT_error1}
	\|\x^{(n)} - \x\|_1 \leq \alpha^n \|\x\|_1 + \beta\|\A\x_{\support^c}\|_1 + \beta\|\noise\|_1.
\end{equation}
Using the fact that $\A$ is a binary matrix with $d$ ones per column we have $\|\A\x_{\support^c}\|_1 \leq d\|\x_{\support^c}\|_1$. We also have $\|\noise\|_1 \geq \alpha^n \|\x\|_1$ when  $n \geq \log\left(\frac{\|\x\|_1}{\|\noise\|_1}\right) / \log \left(\frac{1}{\alpha}\right)$. Applying these bounds to \eqref{eqn:MEIHT_error1} leads to
\begin{equation}
\label{eqn:MEIHT_error2}
\|\x^{(n)} - \x\|_1 \leq \beta d\|\x_{\support^c}\|_1 + \left(1+\beta\right)\|\noise\|_1
\end{equation}
for  $n = \left\lceil\log\left(\frac{\|\x\|_1}{\|\noise\|_1}\right) / \log \left(\frac{1}{\alpha}\right) \right\rceil$. This is equivalent to \eqref{eqn:lim_MEIHT_error} with $c_1 = \beta d$, $c_2 = 1+\beta$, $\x^{(n)} = \sol$ for the given $n$, and $\sigma_{\fG_G}(\x)_1 = \|\x_{\support^c}\|_1$ because $\x_{\support}$ is the best $\fG_G$-term approximation to the $\fG_G$-sparse $\x$. This completes the proof.\qed
\end{proof}

The runtime complexity of MEIHT still depends on the median operation and the complexity of the projection onto the model. However, as observed in \cite{bah2014model}, the projection onto the model is the dominant operation of the algorithm. Therefore, the complexity of MEIHT is of the order of the complexity of the projection onto the model. In the case of overlapping group models with bounded treewidth, MEIHT achieves a polynomial runtime complexity as shown in Proposition \ref{pro:meiht_time} below. On the other hand, when the treewidth of the group model is unbounded MEIHT can still be implemented by using the Bender's decomposition procedure in Section \ref{sec:exact} for the projection which may have an exponential runtime complexity.

\begin{proposition}\label{pro:meiht_time}
	The runtime of MEIHT is $\bigO((N+M)(w_T^2 5^{w_T} G^2 N) \bar n + (N+M)(2^{\bigO(w_T^3)}+w_T^2))$ for the $\fG_G$-group-sparse model with bounded treewidth $w_T$, 
where $\bar n$ is the number of iterations, $M$ is the number of groups, $G$ is the group budget and $N$ is the signal dimension. 
\end{proposition}

\begin{proof}
Before we start the MEIHT procedure we have to calculate a nice tree decomposition of the incidence graph of the group model. This can be done in $\mathcal O ((N+M)2^{\bigO(w_T^3)}+w_T^2)$. Then in each of the iterations of the MEIHT we have to solve the exact projection onto the model which is the dominant operation of the MEIHT. Since the projection on the group model with bounded treewidth $w_T$ can be done through the dynamic programming algorithm that runs in $\bigO((N+M)(w_T^2 5^{w_T} G^2 N))$, as proven in Section \ref{sec:treewidth}, this proves the result. \qed
\end{proof}

\begin{remark}
	The convergence results above hold when MEIHT is modified appropriately to solve the standard $K$-sparse and $G$-group-sparse problem with groups having bounded treewidth, where the projection becomes $\mathcal{P}_{\fG_{G,K}}$. However, in this case the runtime complexity in each iteration grows by a factor of $\bigO(K)$, as indicated in Remark \ref{rem:kplusG}.
\end{remark}

\subsection{Exact Projection for General Group Models}\label{sec:exact}

In this section we consider the most general version of group models, i.e. $\mathfrak{G} = \{\cG_1,\ldots,\cG_M\}$ is an arbitrary set of groups and $G\in [M]$ and $K\in [N]$ are given budgets. We study the structured sparsity model $\mathfrak G_{G,K}$ introduced in Section \ref{sec:Prel_GroupStructures}. Here additional to the number $G$ of groups that can be selected we bound the number of indices to be selected in these groups by $K$ (i.e. we consider a group-sparse model with an additional standard sparsity constraint). Note that setting $K = N$ reduces this model $\mathfrak G_{G,K}$ to the general group model $\mathfrak G_{G}$.

If we want to apply exact projection recovery algorithms like the Model-IHT and MEIHT to group models, iteratively the model projection problem has to be solved, i.e. in each iteration for a given signal $\mathbf \x\in\RR^N$ we have to find the closest signal $\mathbf{\hat \x}$ which has support in the model $\mathfrak G_{G,K}$. In this section we will derive an efficient procedure based on the idea of Benders' Decomposition to solve the projection problem. This procedure is analysed and implemented in Section \ref{sec:computations}. 


It has been proved that the \MS~problem for group models without a sparsity condition on the support is NP-hard \cite{baldassarre2016group}. Therefore the projection problem on the more general model $\mathfrak G_{G,K}$ is NP-hard as well. The latter problem can be reformulated by the integer programming formulation
\begin{equation}\label{eqn:generalprob}
\begin{aligned}
\max \ & \bf w^\top \bu \\
s.t. & \sum_{i=1}^{N} u_i \le K \\
& \sum_{j=1}^{M} v_j \le G \\
& u_i\le \sum_{j:i\in \cG_j} v_j \quad \forall i=1,\ldots , N \\
& \bu\in\{ 0,1\}^N, \ \bv\in \{ 0,1\}^M .
\end{aligned}
\end{equation}
Here $\w$ are the squared entries of the signal, the $\bv$-variables represent the groups and the $\bu$-variables represent the elements which are selected. Note that by choosing $K=N$ we obtain the projection problem for classical group models $\mathfrak G_G$.

To derive an efficient algorithm for the projection problem we use the concept of Benders' Decomposition which was already studied in \cite{benders1962partitioning,geoffrion1972generalized}. The idea of this approach is to decompose Problem \ref{eqn:generalprob} into a master problem and a subproblem. Then iteratively the subproblem is used to derive feasible inequalities for the master problem until no feasible inequality can be found any more. This procedure has been applied to Problem \ref{eqn:generalprob} without the sparsity constraint on the $\bu$-variables in \cite{cordeau2018benders}. The following results for the more general Problem \ref{eqn:generalprob} are based on the the idea of Benders' Decomposition and extend the results in \cite{cordeau2018benders}.

First we can relax the $\bu$-variables in the latter formulation without changing the optimal value, i.e. we may assume $\bu\in [0,1]^N$. We can now reformulate \eqref{eqn:generalprob} by
\begin{equation}\label{eqn:BendersStartProblem}
\begin{aligned}
\max \ & \mu \\
s.t. \quad & \mu \le \max_{\bu\in P(\bv)} {\bf w^\top \bu} \\
& \sum_{j=1}^{M} v_j \le G \\
& \bv\in \{ 0,1\}^M
\end{aligned}
\end{equation}
where $P(\bv)=\{\bu\in [0,1]^N \ | \ \sum_{i=1}^{N} u_i \le K, \ u_i\le \sum_{j:i\in \cG_j} v_j, \ i=1,\ldots , N\}$. Replacing the linear problem $\max_{\bu\in P(\bv)} {\bf w^\top \bu}$ in \eqref{eqn:BendersStartProblem} by its dual formulation, we obtain
\begin{equation}\label{eqn:generalprobrefform}
\begin{aligned}
\max \ & \mu \\
s.t. \quad & \mu \le \min_{\alpha,{\boldsymbol \beta},{\boldsymbol\gamma}\in P_D} \alpha K + \sum_{i=1}^{N}\beta_i\left( \sum_{j:i\in \cG_j} v_j \right) + \sum_{i=1}^{N} \gamma_i \\
& \sum_{j=1}^{M} v_j \le G \\
& \bv\in \{ 0,1\}^M
\end{aligned}
\end{equation}
where $P_D=\{\alpha,{\boldsymbol\beta},{\boldsymbol \gamma}\ge 0 \ | \ \alpha + \beta_i + \gamma_i\ge w_i \ i=1,\ldots , N\}$ is the feasible set of the dual problem. Since $P_D$ is a polyhedron and the minimum in \eqref{eqn:generalprobrefform} exists, the first constraint in \eqref{eqn:generalprobrefform} holds if and only if it holds for each vertex $(\alpha^l,{\boldsymbol \beta}^l,{\boldsymbol \gamma}^l)$ of $P_D$. Therefore Problem \eqref{eqn:generalprobrefform} can be reformulated by
\begin{equation}\label{eqn:masterprob}
\begin{aligned}
\max \ & \mu \\
s.t. \quad & \mu \le \alpha^l K + \sum_{i=1}^{N}\beta_i^l\left( \sum_{j:i\in \cG_j} v_j \right) + \sum_{i=1}^{N} \gamma_i^l \quad l=1,\ldots ,t \\
& \sum_{j=1}^{M} v_j \le G \\
& \bv\in \{ 0,1\}^M
\end{aligned}
\end{equation}
where $(\alpha^1,{\boldsymbol \beta}^1,{\boldsymbol \gamma}^1),\ldots ,(\alpha^t,{\boldsymbol \beta}^t,{\boldsymbol \gamma}^t)$ are the vertices of $P_D$. Each of the constraints 
\[
\mu \le \alpha^l K + \sum_{i=1}^{N}\beta_i^l\left( \sum_{j:i\in \cG_j} v_j \right) + \sum_{i=1}^{N} \gamma_i^l
\]
for $l=1,\ldots ,t$ is called \textit{optimality cut}.

The idea of Benders' Decomposition is, starting from Problem \eqref{eqn:masterprob} containing no optimality cut (called the \textit{master problem}), to iteratively calculate the optimal $({\bv}^*,\mu^*)$ and then find a optimality cut which cuts off the latter optimal solution. In each step the most violating optimality cut is determined by solving
\begin{equation}\label{eqn:subprob}
\begin{aligned}
\min & \ \ \alpha K + \sum_{i=1}^{N}\beta_i\left( \sum_{j:i\in \cG_j} v_j^* \right) + \sum_{i=1}^{N} \gamma_i \\
s.t. \quad & \alpha + \beta_i + \gamma_i\ge w_i \ \ i=1,\ldots , N \\
& \alpha,{\boldsymbol \beta},{\boldsymbol \gamma}\ge 0,
\end{aligned}
\end{equation}
for the actual optimal solution ${\bv}^*$. If the optimal solution fulfills
\[
\mu^* > \alpha^* K + \sum_{i=1}^{N}\beta_i^*\left( \sum_{j:i\in \cG_j} v_j^* \right) + \sum_{i=1}^{N} \gamma_i^*,
\]
then the optimality cut related to the optimal vertex $(\alpha^*,\boldsymbol\beta^*,\boldsymbol\gamma^*)$ is added to the master problem. This procedure is iterated until the latter inequality is not true any more. The last optimal $\bv^*$ must then be optimal for Problem \eqref{eqn:generalprob} since the first constraint in \eqref{eqn:generalprobrefform} is then true for $\bv^*$.

If we use the latter Benders' Decomposition approach it is desired to use fast algorithms for the master problem \eqref{eqn:masterprob} and the subproblem \eqref{eqn:subprob} in each iteration. By the following lemma an optimal solution of the subproblem can be easily calculated.
\begin{lemma}\label{lem:subproblem}
For a given solution $\bv\in\{ 0,1\}^M$ we define $I_{\bv}:=\left\{ i=1,\ldots ,N\ | \ \sum_{j:i\in \cG_j}v_j > 0\right\}$ and $I_{\bv}^K$ by the indices of the $K$ largest values $w_i$ for $i\in I_{\bv}$. An optimal solution of Problem \eqref{eqn:subprob} is then given by $(\alpha^*,\boldsymbol \beta^*,\boldsymbol \gamma^*)$ where $\alpha^* = \max_{i\in I_{\bv}\setminus I_{\bv}^K} w_i$ and 
\[
(\beta_i^*,\gamma_i^*) =
\begin{cases}
(w_i,0) & \text{ if } i\in [N]\setminus I_{\bv}\\
(0,0) & \text{ if } i\in I_\bv\setminus I_{\bv}^K \\
(0, w_i-\alpha^*) & \text{ if } i\in I_{\bv}^K \\
\end{cases}
\]
\end{lemma}
\begin{proof}
Note that for a given ${\bv}^*$ the dual problem of subproblem \eqref{eqn:subprob} is
\begin{equation}\label{eqn:primalproblem}
\begin{aligned}
\max \ &{\bf w}^\top {\bu}\\
s.t. \quad &\sum_{i=1}^{N} u_i \le K \\
&u_i\le \sum_{j:i\in \cG_j} v_j^* \ \ i=1,\ldots , N\\
& {\bu}\in [0,1]^N . 
\end{aligned}
\end{equation}
It is easy to see that there exists an optimal solution $\bu^*$ of Problem \eqref{eqn:primalproblem} where $u_i=1$ if and only if $i\in I_{\bv}^K $ and $u_i^*=0$ otherwise. We will use the dual slack conditions
\begin{equation}\label{eqn:slackcondition}
\left(\sum_{j:i\in \cG_j} v_j^* - u_i^*\right)\beta_i = 0 \text{ and } \left(1 - u_i^*\right)\gamma_i = 0 \ i=1,\ldots ,N,
\end{equation}
to derive the optimal values for $\alpha,\boldsymbol\beta,\boldsymbol\gamma$. We obtain the following $4$ cases:

\textit{Case $1$:} If $u_i^*=0$ and $\sum_{j:i\in \cG_j} v_j^*>0$ , i.e. $i\in I_{\bv}\setminus I_{\bv}^K$, then by conditions \eqref{eqn:slackcondition} we have $\beta_i=\gamma_i=0$. To ensure the constraint $\alpha + \beta_i + \gamma_i\ge w_i$, the value of $\alpha$ must be at least $\max_{i\in I_{\bv}\setminus I_{\bv}^K} w_i$.

\textit{Case $2$:} If $u_i^*=0$ and $\sum_{j:i\in \cG_j} v_j^*=0$, then $i\in[N]\setminus I_{\bv}$ and the objective coefficient of $\beta_i$ in Problem \eqref{eqn:subprob} is $0$. Therefore we can increase $\beta_i$ as much as we want without changing the objective value and therefore we set $\beta_i=w_i$ to ensure the $i$-th constraint $\alpha + \beta_i + \gamma_i\ge w_i$ and set $\gamma_i=0$.

\textit{Case $3$} If $u_i^*=1$, i.e. $i\in I_{\bv}^K$, and $\sum_{j:i\in \cG_j} v_j^*>1$ then by condition \eqref{eqn:slackcondition} $\beta_i=0$. Therefore to ensure the $i$-th constraint $\alpha + \beta_i + \gamma_i\ge w_i$ the value of $\gamma_i$ must be at least $w_i-\alpha$ and since we minimize $\gamma_i$ in the objective function the latter holds with equality.

\textit{Case $4$:} If $u_i^*=1$, i.e. $i\in I_{\bv}^K$, and $\sum_{j:i\in \cG_j} v_j^*=1$ then we cannot use condition \eqref{eqn:slackcondition} to derive the values for $\beta_i$ and $\gamma_i$. Nevertheless in this case both variables have an objective coefficient of $1$ while $\alpha$ has an objective coefficient of $K$. By increasing $\alpha$ by $1$ the objective value increases by $K$. In Cases $1$ and $2$ nothing changes, while for each index in Case $3$ we could decrease $\gamma_i$ by $1$ to remain feasible. But since at most $K$ indices $i$ fulfil the conditions of Case $3$ we cannot improve our objective value by this strategy. Therefore $\alpha$ has to be selected as small as possible in Case $4$, i.e. by Case $1$ we set $\alpha = \max_{i\in I_{\bv}\setminus I_{\bv}^K} w_i$ and to ensure feasibility we set $\gamma_i=w_i-\alpha$. This concludes the proof.\qed
\end{proof}
\begin{theorem}
An optimal solution of subproblem \eqref{eqn:subprob} can be calculated in time $\mathcal O (NM|G_\text{max}|)$.
\end{theorem}
\begin{proof}
The set $I_{\bv}$ can be calculated in time $\mathcal O(NM|G_\text{max}|)$ by going through all groups for each index $i\in [N]$ and checking if the index is contained in one of the groups. To obtain $I_{\bv}^K$ we have to find the $K$-th largest element in $\left\{ w_i \ | \ i\in I_{\bv}\right\}$ which can be done in time $\mathcal O(N)$; see \cite{DBLP:journals/spe/Musser97}. Afterwards we select all values which are larger than the $K$-th largest element which can be done in $\mathcal O (N)$. Assigning all values to $(\alpha,\boldsymbol \beta,\boldsymbol \gamma)$ can also be done in time $\mathcal O (N)$. \qed
\end{proof}

The following theorem states that the masterproblem can be solved in pseudopolynomial time if the number of constraints is fixed. Nevertheless note that the number of iterations of the procedure described above may be exponential in $N$ and $M$.
\begin{theorem}
Problem \eqref{eqn:masterprob} with $t$ constraints can be solved in $\mathcal O (MG(NW)^t)$ where $W=\max_{i\in[N]}w_i$.
\end{theorem}

\begin{proof}
Problem \eqref{eqn:masterprob} with $t$ constraints can be reformulated by
\begin{equation}\label{eqn:robustselection}
\begin{aligned}
\max_{{\bv}\in\left\{ 0,1\right\}^M} \ & \min_{l=1,\ldots ,t}  \alpha^l K + \sum_{i=1}^{N}\beta_i^l\left( \sum_{j:i\in \cG_j} v_j \right) + \sum_{i=1}^{N} \gamma_i^l \\
s.t. \quad & \sum_{j=1}^{M} v_j \le G .
\end{aligned}
\end{equation}
The latter problem is a special case of the robust knapsack problem with discrete uncertainty (sometimes called robust selection problem) with additional uncertain constant. In \cite{buchheim2018robust} the authors mention that the problem with an uncertain constant is equivalent to the problem without such a constant. Furthermore using the result in \cite{kouvelis2013robust} Problem \eqref{eqn:robustselection} can be solved in $\mathcal O (MGC^t)$ where 
\[C:=\max_{l=1,\ldots t} \alpha^l + \sum_{i=1}^{N}\beta_i^l + \sum_{i=1}^{N}\gamma_i^l .\]
Since for all solutions $(\alpha^l,\boldsymbol \beta^l,\boldsymbol\gamma^l)$ generated in Lemma \ref{lem:subproblem} it holds that $\alpha^l, \beta_i^l, \gamma_i^l \le \max_{i\in [N]} w_i$ we have $C\le (2N+1)W$ which proves the result.\qed
\end{proof}

%

\section{Algorithms with Approximation Projection Oracles}\label{sec:low-freq}

As mentioned in the previous sections solving the \MS~problem is \NP-hard in general. Therefore to use classical algorithms as the Model-IHT or the MEIHT we have to solve an \NP-hard problem in each iteration. To tackle problems like this the authors in \cite{hegde2015fast,hegde2015approximation} introduced an algorithm called \textit{Approximate Model-IHT} (AM-IHT) which is based on the idea of the classical IHT but which does not require an exact projection oracle (see Section \ref{sec:preliminaries_AMIHT}). Instead the authors show that under certain assumptions on the measurement matrix a signal can be recovered by just using two approximation variants of the projection problem which they call head- and tail-approximation (for further details again see Section \ref{sec:preliminaries_AMIHT}). 

In this section we apply the latter results to group models of bounded frequency: group models where the maximum number of groups an element is contained in is bounded by some number $f$. 
Note that from Theorem~\ref{thm:hardness-on-grids}~we know that \MS~is NP-hard for group models of frequency 4. 
A particularly interesting case of such group structures is the graphic case, where each element is contained in at most two groups. 
Understanding this case was proposed as an open problem by Baldassarre et al.~\cite{baldassarre2016group}.

Furthermore we apply the theoretical results derived in \cite{hegde2015fast,hegde2015approximation} to group models and show that the number of required measurements compared to the classical structured sparsity case increases by just a constant factor. In Section \ref{sec:computations} we will computationally compare the AM-IHT to the Model-IHT and the MEIHT on interval groups. 

\subsection{Head- and Tail-Approximations for Group Models with Low Frequency}
In this section we derive head- and tail-approximations for the case of group models with bounded frequency. We first recall the definition of head- and tail-approximations for the case of group models. Assume we have given a group model $\fG$ together with $G\in \mathbb N$.

Given a vector $\x$, let $\mathcal H$ be an algorithm that computes a vector $\mathcal H(\x)$ with support in $\fG_{G'}$ for some $G'\in \mathbb N$. Then, given some $\alpha \in \mathbb R$ (typically $\alpha < 1$) we say that $\mathcal H$ is an \emph{$(\alpha,G,G',p)$-head approximation} ~if
\begin{equation}\label{eqn:head-condition_groups}
|| \mathcal H(\x) ||_p \ge \alpha \cdot || \x_{\mathcal S}||_p \mbox{ for all } \mathcal S \in \fG_G.
\end{equation}
In other words, $\mathcal H$ uses $G'$ many groups to cover at least an $\alpha$-fraction of the maximum total weight covered by $G$ groups. Note that $G'$ can be chosen larger than $G$.

Moreover, let $\mathcal T$ be an algorithm which computes a vector $\mathcal T(\x)$ with support in $\fG_{G}$. Given some $\beta \in \mathbb R$ (typically $\beta > 1$) we say that $\mathcal T$ is a \emph{$(\beta,G,G',p)$-tail approximation} if
\begin{equation}\label{eqn:tail-condition_groups}
|| \x - \mathcal T(\x) ||_p \le \beta \cdot || \x - \x_{\mathcal S}||_p \mbox{ for all } \mathcal S \in \fG_G.
\end{equation}
This means that $\mathcal T$ may use $G'$ many groups to leave at most a $\beta$-fraction of weight uncovered compared to the minimum total weight left uncovered by $G$ groups.

In the following we derive the head- and tail-approximation just for the case $p=1$. Note that equivalent approximation procedures can be easily derived for the case $p=2$ by replacing the accuracies $\alpha$ and $\beta$ by $\sqrt{\alpha}$ and $\sqrt{\beta}$ and using the weights $x_i^2$ instead of $|x_i|$ in the latter proofs. We will first present a result which implies the existence of a head approximation.

\begin{theorem}[Hochbaum and Pathria~\cite{hochbaum1998analysis}]\label{thm:hochbaum}
For each $\epsilon > 0$ there exists an $((1-\epsilon),G, \lceil G \log_2 (1/\epsilon) \rceil,1)$-head approximation running in polynomial time.
\end{theorem}
The algorithm derived in \cite{hochbaum1998analysis} was designed to solve the \MWC ~problem and is based on a simple greedy method. The idea is to iteratively select the group which covers the largest amount of uncovered weight. It is proven by the authors that if you are allowed to select enough groups, namely $\lceil G \log_2 (1/\epsilon)\rceil$, then the optimal value is approximated up to an accuracy of $(1-\epsilon)$. The greedy procedure is given in Algorithm \ref{alg:head_approx}. Note that for a given signal $\x$ and a group $\mathcal G\in\mathfrak G$ we define
\[
w(\mathcal G ):= \sum_{i\in \mathcal G} |x_i| . 
\]
\begin{algorithm}\caption{Head-Approximation}\label{alg:head_approx}
\inp $\x\in\mathbb R^N$, $G\in\mathbb N$, $\mathfrak G$, $\epsilon>0$\\
\out $\mathcal H(\x, G,\mathfrak G, \epsilon)$, a head-approximation to $\x$
\begin{algorithmic}[1]
\State $\z\gets 0$
\For{$i=1\ldots G$}
\State Sort the groups in $\mathfrak G$ by decreasing total weight: $w(\mathcal G_{i_1})\ge\ldots \ge w(\mathcal G_{i_M})$. 
\State $z_j \gets x_j \ \forall \ j\in \mathcal G_{i1}$
\State Delete the indices in $\mathcal G_{i1}$ from all groups in $\mathfrak G$.
\EndFor
\State Return: $\mathcal H(\x, G,\mathfrak G, \epsilon) \gets \z$
\end{algorithmic}
\end{algorithm}

Next we derive a tail-approximation for our problems based on the idea of LP rounding. Note that in contrast to the head-approximation, the run-time bound of the following tail-approximation depends on the frequency of the group model.

\begin{theorem}\label{thm:bounded-frequency}
Suppose the frequency of the group model is $f$. For any $\epsilon > 0$ and $\kappa = (1+\epsilon^{-1}) f$ there exists an $((1+\epsilon),G,\kappa G,1)$-tail approximation running in polynomial time.
\end{theorem}
\begin{proof}
Given a signal $\x \in \mathbb R^N$, we define $\w = (|x_i|)_{i \in [N]}$.
We consider the following linear relaxation of the \MS~problem
\begin{equation}\label{eqn:generalprob-linear-relax}
\begin{aligned}
\max \ & \w^\top \bu \\
s.t. & \sum_{i \in [M]} v_i = G \\
& u_j \le \sum_{j \in \cG_i} v_i \quad \mbox{for all } j \in [N] \\
& \bu \in [0,1]^N, \ \bv \in [0,1]^M .
\end{aligned}
\end{equation}
Consider an optimal solution $(\bu,\bv)$ of \eqref{eqn:generalprob-linear-relax}.
We compute a group cover $\cS \subseteq \mathcal \fG$ by
\[
\cS = \{ \cG_i \in \fG : v_i \ge \kappa^{-1}\}.
\]
Note that $\cS$ contains at most $\kappa G$ many groups, since
\[
|\cS| \le \kappa \sum_{i=1}^M v_i = \kappa G.
\]
It remains to show that $\cS$ is a tail approximation.
To this end let $R$ be the set of indices only barely covered by $v$, i.e.
\[
R = \{j \in [N] : \sum_{j \in \cG_i} v_i \le (1+\epsilon^{-1})^{-1}\}.
\]
Note that 
\begin{equation}\label{eqn:S-is-good}
\cS \mbox{ covers every element } j \in [N] \setminus R,
\end{equation} 
since $j \notin R$ implies 
\[\sum_{j \in \cG_i} v_i > (1+\epsilon^{-1})^{-1}\]
and hence 
\[v_i \ge \tfrac{(1+\epsilon^{-1})^{-1}}{f} = \kappa^{-1}\]
for some $i$ with $j \in \cG_i$. Moreover, note that 
\[
u_j \le \sum_{j \in \cG_i} v_i \le (1+\epsilon^{-1})^{-1}
\]
and hence
\begin{equation}\label{eqn:rest-elts}
\frac{1-u_j}{1-(1+\epsilon^{-1})^{-1}} \ge 1 \mbox{ for each } j \in R.
\end{equation}
We obtain the inequalities
\begin{align*}
\| x - x_{\bigcup \cS}\|_1 = \sum_{i\in [N] \setminus \bigcup \cS} w_i \le \sum_{i\in R} w_i &\le \sum_{j \in R} \frac{1-u_j}{1-(1+\epsilon^{-1})^{-1}} w_j \\
&\le (1+\epsilon) \sum_{j \in R} w_j (1-u_j) \\
&\le (1+\epsilon) (\w^\top\boldsymbol{1} - \w^\top\bu),
\end{align*}
where we used~\eqref{eqn:S-is-good} and~\eqref{eqn:rest-elts}. Since $\bu$ is an optimal solution of the relaxed problem \eqref{eqn:generalprob-linear-relax}, we have
\[
(1+\epsilon) (\w^\top\boldsymbol{1} - \w^\top\bu) \le (1+\epsilon) (\w^\top\boldsymbol{1} - \w^\top\bu^*) = (1+\epsilon)\| x-x_{\cS^*}\|_1
\]
where $\cS^*$ is an optimal solution of the \MS~problem and $\bu^*$ the corresponding optimal vector of Problem \eqref{eqn:generalprob-IP}. Therefore the latter procedure is a tail-approximation which completes the proof.\qed
\end{proof}


\subsection{AM-IHT and AM-MEIHT for Group Models}
As in the previous sections for a sensing matrix $\A$ and the true signal $\x$ we have a noisy measurement vector $\obs = \A \mathbf x + \noise$ for some noise vector $\noise$. The task consists in recovering the original signal $\mathbf x$, or a vector close to~$\mathbf x$. Furthermore we have given a group model $\fG$ together with $G\in \mathbb N$ with frequency $f\in\NN$.

In the last section we derived polynomial time algorithms for an $((1-\epsilon),G, \lceil G \log_2 (1/\epsilon) \rceil,2)$-head approximation and an $((1+\epsilon),G,(1+\epsilon^{-1}) f G,2)$-tail approximation. Note that we can use any $G$ here. Using the results in Section \ref{sec:preliminaries_AMIHT}, we obtain convergence of the AM-IHT for group models if $\mathcal T$ is an $((1+\epsilon),G,G_T,2)$-tail approximation, $\mathcal H$ is an $((1-\epsilon),G_T+G, G_H,2)$-head approximation, where $G_T:=(1+\epsilon^{-1}) f G$ and $G_H:= \lceil (G_T+G) \log_2 (1/\epsilon)\rceil$, and the sensing matrix $\A$ has $\mathfrak G_{\tilde G}$-RIP with $\tilde G = G+G_T+G_H$. Note that $\tilde G\in \mathcal O (G)$ for fixed accuracy $\epsilon>0$ and frequency $f$. Furthermore $|\mathfrak G_{\tilde G}|\in \mathcal O (M^{cG})$ for a constant $c$. Using the bound \eqref{eqn:measurementboundgroupmodels} we obtain that the number of required measurements for a sub-Gaussian random matrix $\A$ having the $\mathfrak G_{\tilde G}$-RIP with high probability is
\[
m=\mathcal O\left( \delta^{-2} Gg_{\text{max}} \log(\delta^{-1}) + cG\log(M) \right)
\]
which differs by just a constant factor from the number of measurements required in the case of exact projections (see Section \ref{sec:exactprojections}). Under condition \eqref{eqn:condition_convergenceofAMIHT} convergence of the AM-IHT is ensured.
\subsection{Extension to within-group sparsity and beyond} 
The head and tail approximation approach can be extended far beyond the standard group-sparsity model.
It still works even if we are considering $K$-sparse and $G$-group-sparse (i.e. $\fG_{G,K}$-sparse) vectors in our model, for example.

The reason is that the \MS~can be head approximated to a constant even in this case. Formally, if we search for the $K$ weight maximal elements covered by $G$ many groups, we are maximizing a submodular function subject to a knapsack constraint.\footnote{Actually, we are maximizing a submodular function subject to a uniform matroid constraint which is a simpler problem.}
This is known to be approximable to a constant factor (cf.~Kulik et al.~\cite{kulik2009maximizing,kulik2013approximations} and related work).
Suppose we delete the covered elements and run such an approximation algorithm again. Then, after a constant number of steps, we obtain a collection of groups and elements such that the total weight of the elements is at least an $(1-\epsilon)$-fraction of the total weight that a $K$-sparse and $G$-group-sparse solution could ever have.
Moreover, the sparsity budgets are exceeded only by a constant factor each.

Similarly, the analysis given in the proof of Theorem~\ref{thm:bounded-frequency} works even if we impose sparsity on both groups and elements.
Again, we obtain a solution that has a $(1+\epsilon)$-tail guarantee whose support exceeds that of a $G$-group-sparse $K$-sparse vector by at most a constant factor if we assume a bounded frequency.
This leads to the positive consequences detailed above.

More generally, any knapsack constraints on groups and elements can be handled, leading to head and tail approximations in the case when there are non-uniform sparsity budgets on the groups and elements.
However the corresponding head approximations are rather involved, and certainly much less efficient than the simple greedy procedure proposed by the Hochbaum and Pathria algorithm~\cite{hochbaum1998analysis}.

\section{Computations}\label{sec:computations}
In this section we present the computational results for the Model-IHT, MEIHT, AM-IHT and AM-EIHT presented in Section \ref{sec:prelim} for block-group instances.
Precisely, we study block-groups, i.e. each group $\mathcal G \in\mathfrak G$ is a set of sequenced indices, $\mathcal G=[s,t]\cap [N]$, where $1\le s<t\le N$ and each group has the same size $|\mathcal G|=l$. For a given dimension $N$ we generate blocks of size $l=\lfloor 0.02N\rfloor$. We consider two types of block-models, one where the successive groups overlap in $\lfloor\frac{l-1}{2}\rfloor$ items and another where they overlap in $l-1$ items. Note that the frequency is then given by $f=2$ or $f=l$ respectively.

We run all algorithms for random signals $\x\in\RR^N$ in dimension $N\in\left\{100,200,\ldots ,900\right\}$. For each dimension we vary the number of measurements $m\in\left\{ 20,40,.., N\right\}$ and generate $20$ random matrices $\A\in\RR^{m\times N}$ each together with a random signal $\x\in\RR^N$. We assume there is no noise, i.e. $\noise=0$. For a given group model $\mathfrak G$ the support of the signal $\x$ is determined as the union of $G$ randomly selected groups. The components of $\x$ in the support are calculated as identical and independent draws from a standard Gaussian distribution while all other components are set to $0$. Our computations are processed for two classes of random matrices, Gaussian matrices and expander matrices as described in Section~\ref{sec:prelim}. The Gaussian matrices are generated by drawing identical and independent values from a standard Gaussian distribution for each entry of $\A$ and afterwards normalizing each entry by $\frac{1}{\sqrt{m}}$. The expander matrices are generated by randomly selecting $d=\lfloor2\log(N)/\log(G l)\rfloor$ indices in $\left\{ 1,\ldots , m\right\}$ for each column of the matrix. The choice of $d$ is motivated by the choice in \cite{bah2014model}. Each of the algorithms is stopped if either the number of iterations reaches $1000$ or if for any iteration $i+1$ we have $\|\x^{(i+1)}-\x^{(i)}\|_p<10^{-5}$. For the error in each iteration we use $p=1$ for the calculations corresponding to the expander matrices and $p=2$ for the Gaussian matrices. After the determination of the algorithm we calculate the relative error of the returned signal $\hat \x$ to the true signal $\x$, i.e. we calculate
\[
\frac{\| \x-\hat \x\|_p}{\| \x\|_p}.
\]
Again we use $p=1$ for the calculations corresponding to the expander matrices and $p=2$ for the Gaussian matrices. We call a signal \textit{recovered} if the relative error is smaller than $10^{-5}$. For the AM-IHT and the AM-EIHT the approximation accuracy of the head- and the tail approximation algorithms are set to $\alpha = 0.95$ and $\beta=1.05$.

For the exact signal approximation problem which has to be solved in each iteration of the Model-IHT and the MEIHT we implemented the Benders' decomposition procedure presented in Section \ref{sec:exact}. To this end the master problem is solved by CPLEX 12.6 while each optimal solution of the subproblem is calculated using the result of Lemma \ref{lem:subproblem}. Regarding the AM-IHT, for the head-approximation we implemented the greedy procedure given in Algorithm \ref{alg:head_approx} while for the tail-approximation we implemented the procedure of Theorem \ref{thm:bounded-frequency}. Again the LP in the latter procedure is solved by CPLEX 12.6. All computations were calculated on a cluster of 64-bit Intel(R) Xeon(R) CPU E5-2603 processors running at 1.60 GHz with 15MB cache.

The results of the computations are presented in the following diagrams. For all instances we measure the smallest number of measurements $m^\#$ for which the median of the relative error to the true signal is at most $10^{-5}$, i.e. the smallest number of measurements for which at least $50\%$ of the signals were recovered. Furthermore we show the average number of iterations and the average time in seconds the algorithms need to successfully recover a signal, given $m^\#$ number of measurements. We stop increasing the number of measurements $m$ if $m^\#$ is reached.

%
\begin{figure}[h!]
\centering
\includegraphics[scale=0.5]{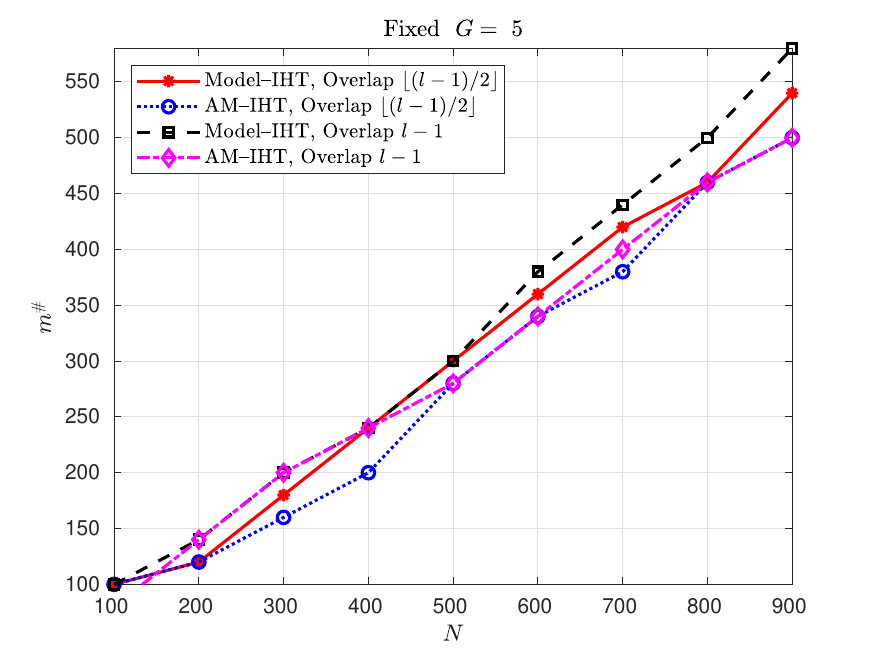}
\includegraphics[scale=0.5]{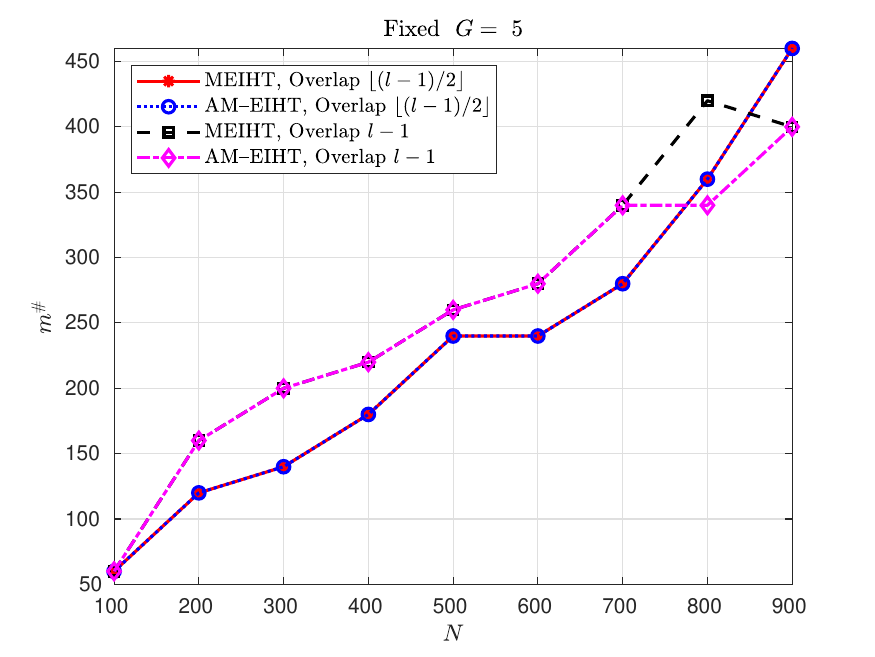}
\caption{Smallest number of measurements $m^\#$ for which the median of the relative error over all random matrices is at most $10^{-5}$.}
\label{fig:VaryNRelError}
\end{figure} 
\begin{figure}[h!]
\centering
\includegraphics[scale=0.5]{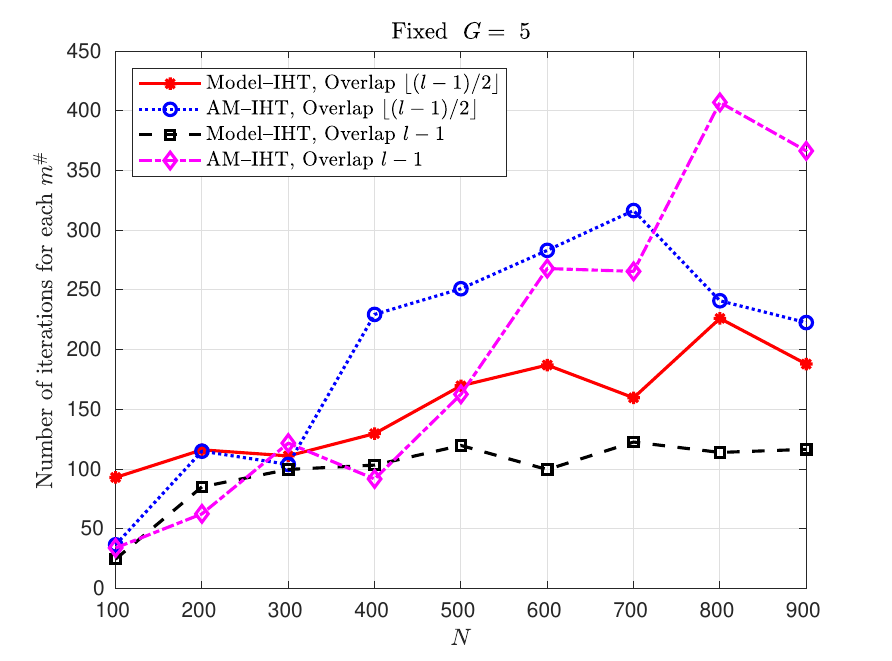}
\includegraphics[scale=0.5]{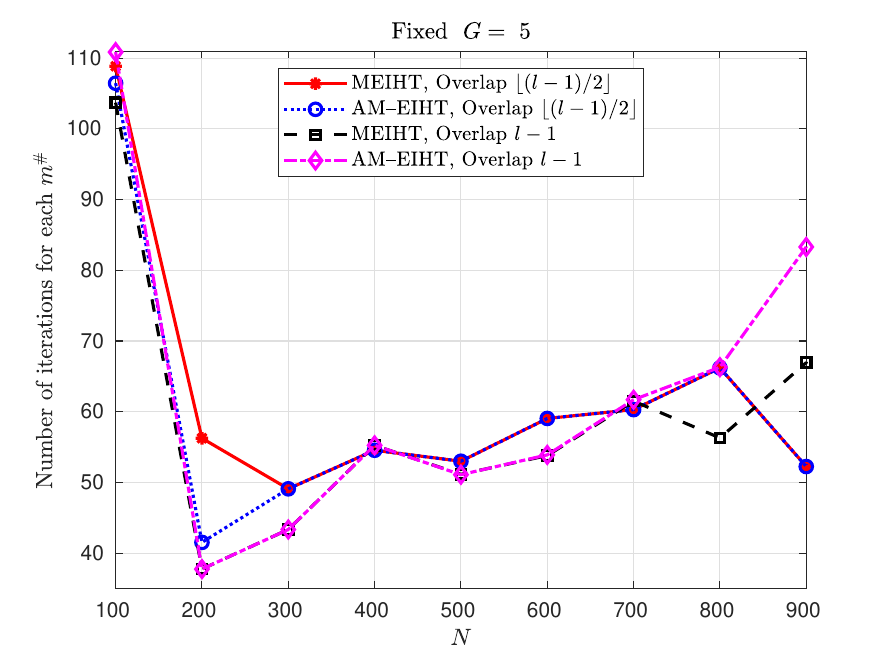}
\caption{Average number of iterations performed by the algorithm to successfully recover the signal with $m^\#$ measurements.}
\label{fig:VaryNIterations}
\end{figure}
\begin{figure}[h!]
\centering
\includegraphics[scale=0.5]{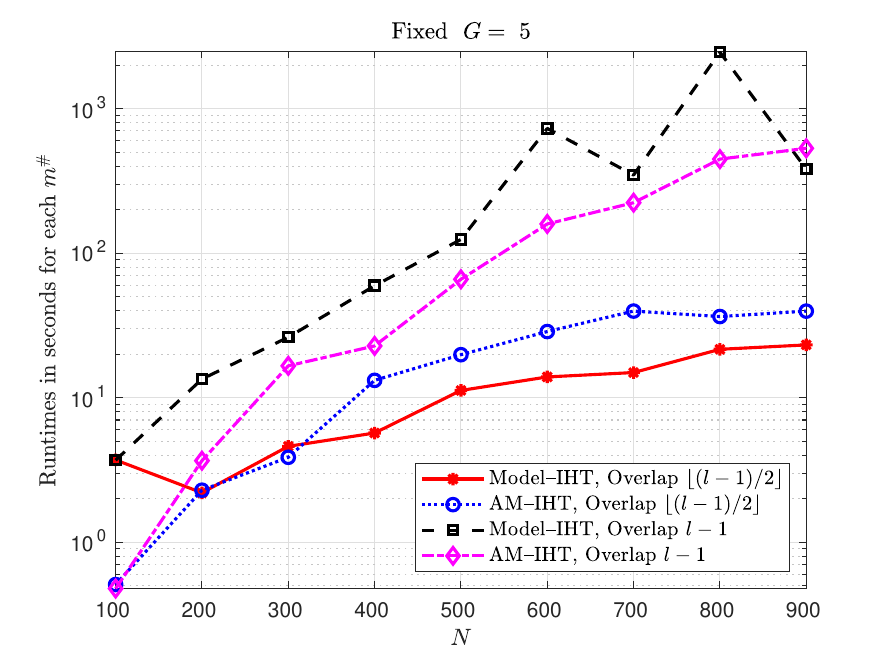}
\includegraphics[scale=0.5]{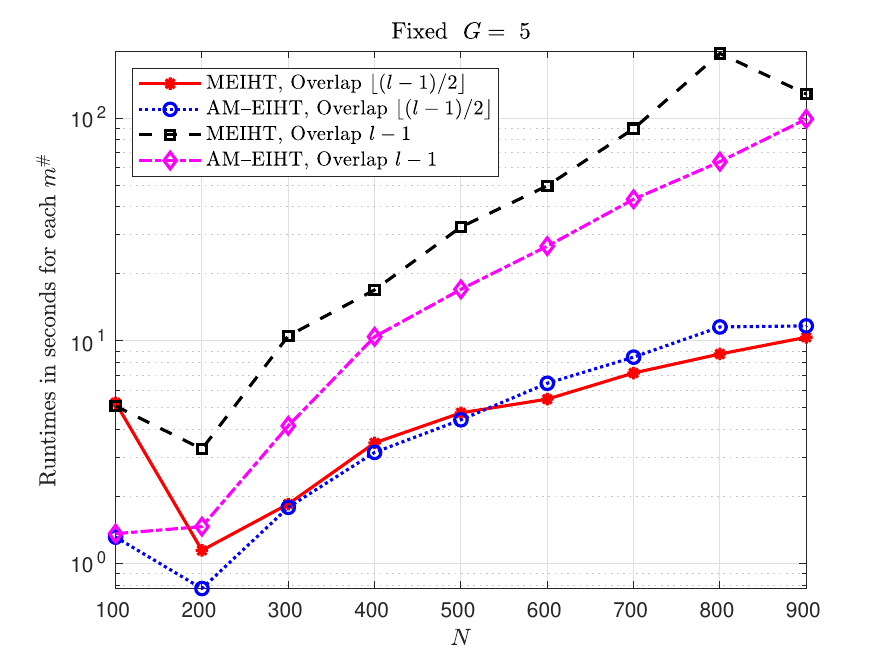}
\caption{Average time in seconds needed by the algorithm to successfully recover the signal with $\bar m$ measurements.}
\label{fig:VaryNTime}
\end{figure}
In Figures \ref{fig:VaryNRelError}, \ref{fig:VaryNIterations} and \ref{fig:VaryNTime} we show the development of $m^\#$, the number of iterations and the runtime in seconds of all algorithms over all dimensions $N\in\left\{100,200,\ldots ,900\right\}$ for block-groups, generated as explained above, with fixed value $G=5$. 

The smallest number of measurements $m^\#$ which leads to a median relative error of at most $10^{-5}$ is nearly linear in the dimension; see Figure \ref{fig:VaryNRelError}. For all algorithms the corresponding $m^\#$ is very close even for different number of overlaps. Nevertheless the number of measurements $m^\#$ is smaller for expander matrices than for Gaussian matrices. Furthermore in the expander case the instances with overlap $\lfloor\frac{l-1}{2}\rfloor$ have a smaller $m^\#$.
The average number of iterations performed by the algorithms fluctuates a lot for the Gaussian case. Here the value increases slowly for the Model-IHT while it increases more rapidly for the AM-IHT. In the expander case the number of iterations is very close to each other for all algorithms and lies between $50$ and $70$ most of the time; see Figure \ref{fig:VaryNIterations}. The drop from $N=100$ to $N=200$ is due to the small value of $d$ when $N=100$. Furthermore the number of iterations is much lower in the expander case.

\begin{figure}[h!]
\centering
\includegraphics[scale=0.5]{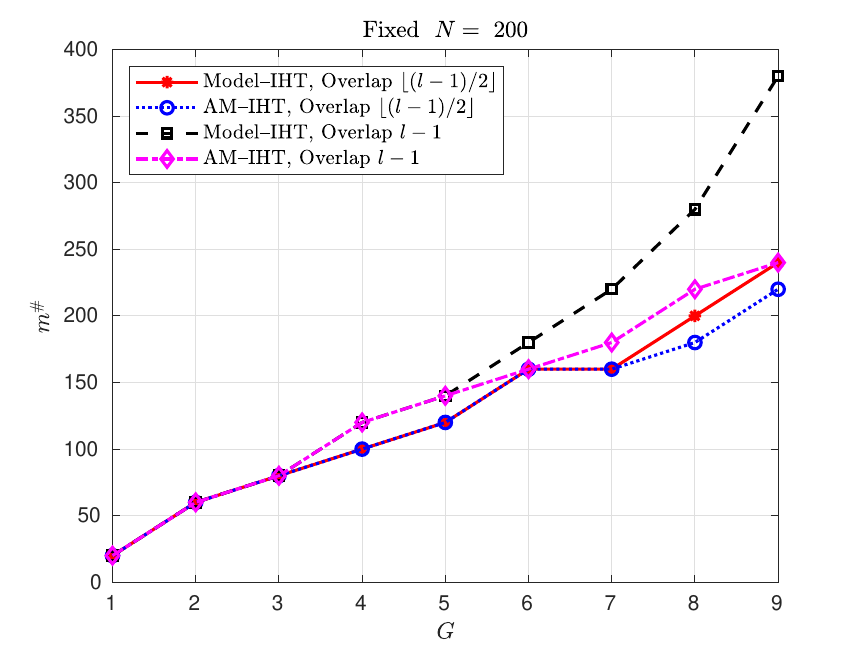}
\includegraphics[scale=0.5]{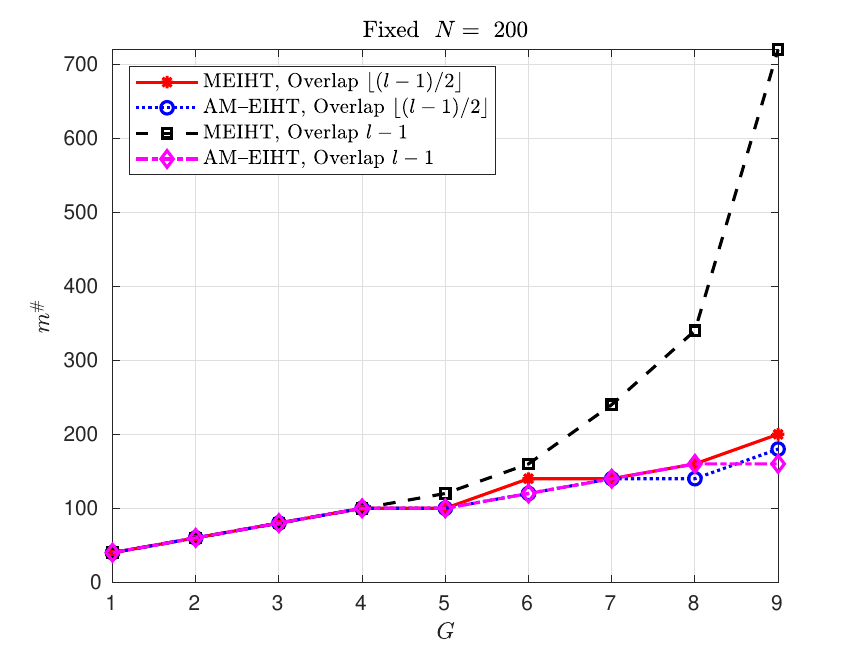}
\caption{Smallest number of measurements $m^\#$ for which the median of the relative error over all random matrices is at most $10^{-5}$.}
\label{fig:VaryGRelError}
\end{figure}
\begin{figure}[h!]
\centering
\includegraphics[scale=0.5]{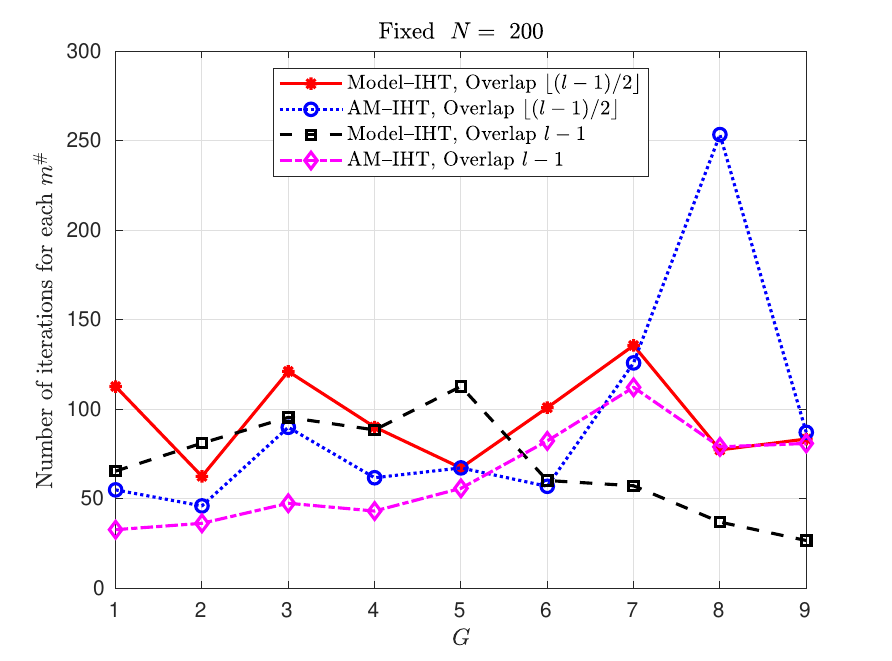}
\includegraphics[scale=0.5]{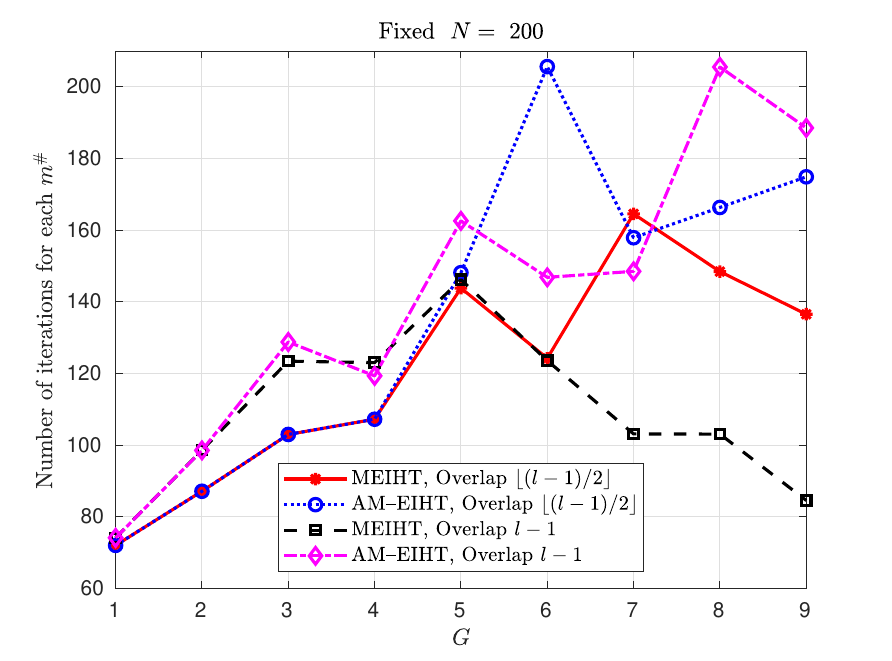}
\caption{Average number of iterations performed by the algorithm to successfully recover the signal with $m^\#$ measurements.}
\label{fig:VaryGIterations}
\end{figure}
\begin{figure}[h!]
\centering
\includegraphics[scale=0.5]{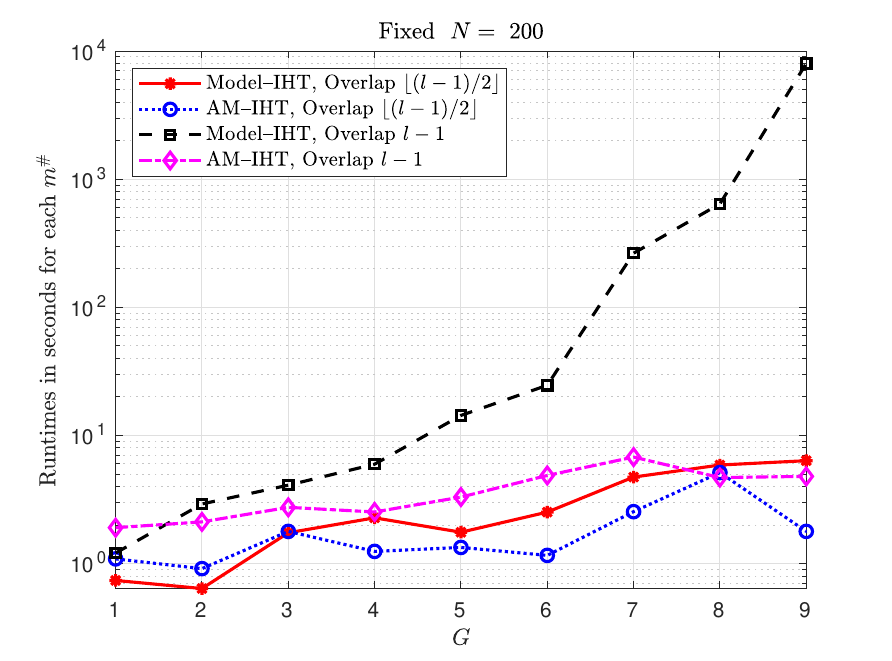}
\includegraphics[scale=0.5]{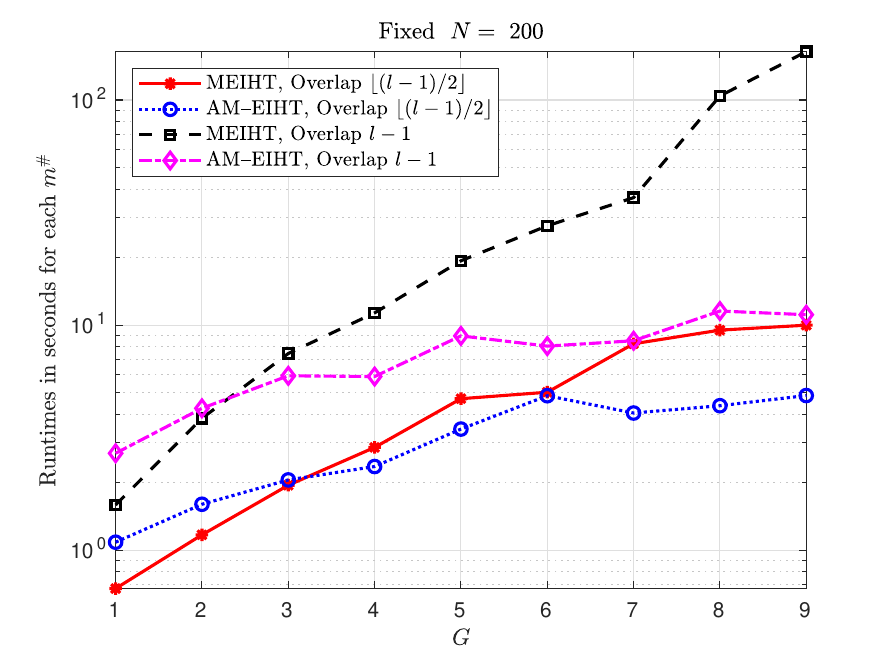}
\caption{Average time in seconds needed by the algorithm to successfully recover the signal with $m^\#$ measurements.}
\label{fig:VaryGTime}
\end{figure}
The average runtime for the Gaussian case is a bit larger than the runtime for the expander case as expected, since operations with dense matrices are more costly than the sparse ones. However, it may be due to the larger number of iterations; see Figure \ref{fig:VaryNTime}. Furthermore the runtime for the instances with $l-1$ overlap is much larger in both cases. Here the AM-IHT (AM-EIHT) is faster than the Model-IHT (MEIHT) for the instances with overlap $l-1$ while it is slightly slower for the others.

In Figures \ref{fig:VaryGRelError}, \ref{fig:VaryGIterations} and \ref{fig:VaryGTime} we show the same properties as above, but for varying $G\in\left\{ 1,2,\ldots ,9\right\}$, a fixed dimension of $N=200$ and $d$ fixed to $7$ for all values of $G$. Similar to Figure \ref{fig:VaryNRelError} the value $m^\#$ seems to be linear in $G$ (see Figure \ref{fig:VaryGRelError}). Just the Model-IHT (MEIHT) for blocks with overlap $l-1$ seems to require an exponential number of measurements in $G$ to guarantee a small median relative error. Here the AM-IHT (AM-EIHT) performs much better. Interestingly the number of iterations does not change a lot for increasing $G$ for Gaussian matrices while it grows for the AM-EIHT in the expander case. This is in contrast to the iteration results with increasing $N$; see Figure \ref{fig:VaryNIterations}. The runtime of all algorithms increases slowly with increasing $G$, except for the IHT for blocks with overlap $l-1$ the runtime explodes. For both instances the AM-IHT (AM-EIHT) is faster than the Model-IHT (MEIHT).

To conclude this section we selected the instances for dimension $N=800$ and $G=5$ and present the development of the median relative error over the number of measurements $m$; see Figure \ref{fig:VarymRelError}.
In the expander case the median relative error decreases rapidly and is nearly $0$ already for $\frac{m}{N}\approx 0.45$. Just for the MEIHT for blocks with overlap $l-1$ the relative error is close to $0$ not until $\frac{m}{N}\approx 0.55$. For the Gaussian case the results look similar with the only difference that a median relative error close to $0$ is reached not until $\frac{m}{N}\approx 0.6$.
\begin{figure}[h!]
\centering
\includegraphics[scale=0.5]{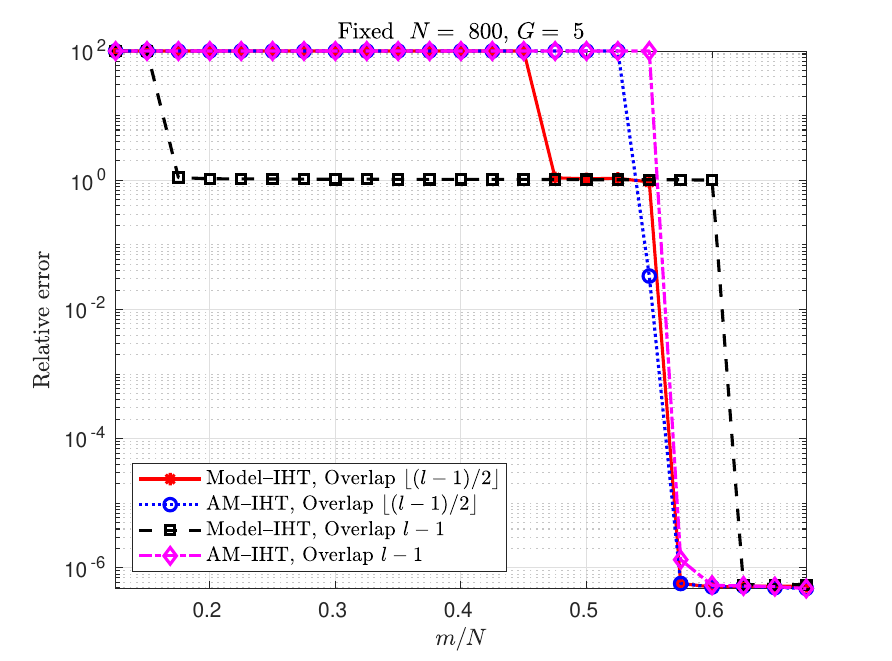}
\includegraphics[scale=0.5]{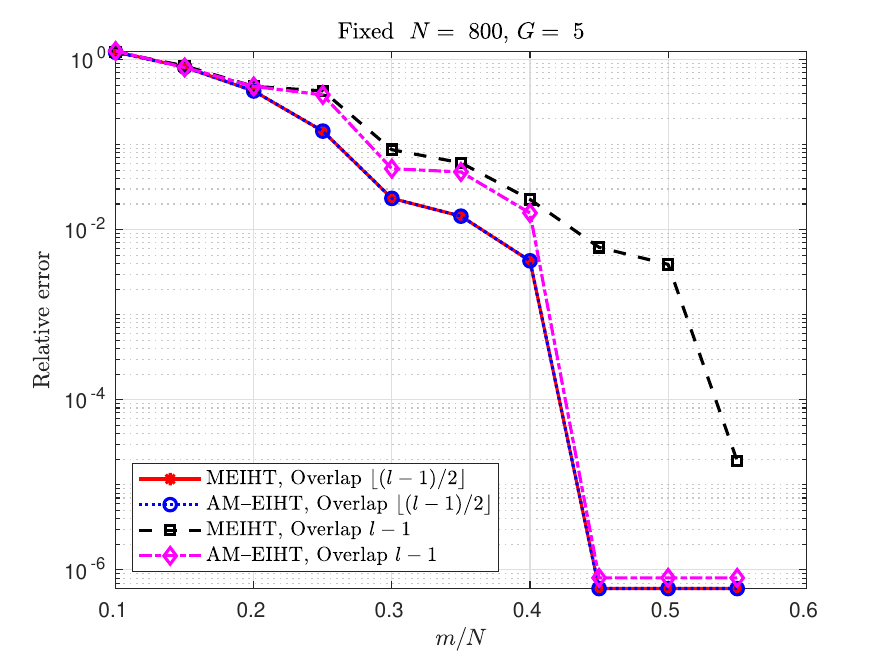}
\caption{Median relative error for instances with dimension $N=800$ and $G=5$. The diagrams are presented in log-scale.}
\label{fig:VarymRelError}
\end{figure}

\subsection{Latent Group Lasso}
In this section we present computational results for the latent group Lasso approach introduced in Section \ref{sec:prelim}. We consider block-groups and all group instances are generated as described in the previous section. We study the $\ell_1/\ell_2$ variant presented in \eqref{eqn:grouplassooverlap} and its $\ell_0$ counterpart presented in \eqref{eqn:grouplassoL0}. Both problems are implemented in CPLEX 12.6. For the $\ell_0$ counterpart we implemented the integer programming formulation \eqref{eqn:grouplassoL0IPFormulation}. For the given random Gaussian matrices $\A\in\RR^{m\times N}$ and its linear measurements $\obs\in\RR^m$ we use $L(\x)=\|\A\x-\obs\|_2^2$ while for expander matrices we use $L(\x)=\|\A\x-\obs\|_1$. The last choice is motivated by our comnputational tests which showed that the $\ell_1$-norm has a better performance for expander matrices. 

We run all algorithms for random signals $\x\in\RR^N$ in dimension $N=200$. The number of measurements is varied in $m\in\left\{ 20,40,.., 2N\right\}$ and we generate $20$ random matrices $\A\in\RR^{m\times N}$ each together with a random signal $\x\in\RR^N$. We assume there is no noise, i.e. $\noise=0$. For a given group model $\mathfrak G$ the support of the signal $\x$ is determined as the union of $G=5$ randomly selected groups. The components of $\x$ in the support are calculated as identical and independent draws from a standard Gaussian distribution while all other components are set to $0$. Our computations are processed for two classes of random matrices, Gaussian matrices and expander matrices generated as described in the previous section. After each calculation we calculate the relative error of the returned signal $\hat \x$ to the true signal $\x$, i.e. we calculate
\[
\frac{\| \x-\hat \x\|_p}{\| \x\|_p}
\]
where we use $p=1$ for the calculations corresponding to the expander matrices and $p=2$ for the Gaussian matrices. Additionally we calculate the pattern recovery error
\[
\frac{1}{2N}\left( |\supp (\x)\setminus \supp (\hat \x)| + |\supp (\hat \x)\setminus \supp (\x)| \right) ,
\]
which was defined in \cite{obozinski2011group}, and the probability of recovery, i.e. the fraction of instances which were successfully recovered. We call a signal \textit{recovered} if the relative error is smaller or equal than $10^{-4}$.

All computations were performed for $\lambda\in\left\{0.25,0.5,\ldots ,5 \right\}$ and $d_{\mathcal G} = 1$ for all groups $\mathcal G\in\mathfrak G$. For each $m\in\left\{ 20,40,.., 2N\right\}$ the $\lambda$ with the best average relative error is calculated and the $\lambda$ which has the best average relative error most often over all $m$ is chosen. For all experiments the optimal value was $\lambda^*=0.25$.

All computations were calculated on a cluster of 64-bit Intel(R) Xeon(R) CPU E5-2603 processors running at 1.60 GHz with 15MB cache.

The results of the computations are presented in Figures \ref{fig:LassoMRE}, \ref{fig:LassoProb}, \ref{fig:LassoPRE}, \ref{fig:LassoG} and \ref{fig:LassoLambda}. For each $m$ we show the median relative error, the probability of recovery, the average pattern recovery error and the average number of selected groups $G$; each value calculated over all $20$ matrices and for $\lambda^*=0.25$. In Figure \ref{fig:LassoLambda} we show for each $m$ the value of $\lambda$ with has the smallest average relative error.

\begin{figure}[h!]
\centering
\includegraphics[scale=0.5]{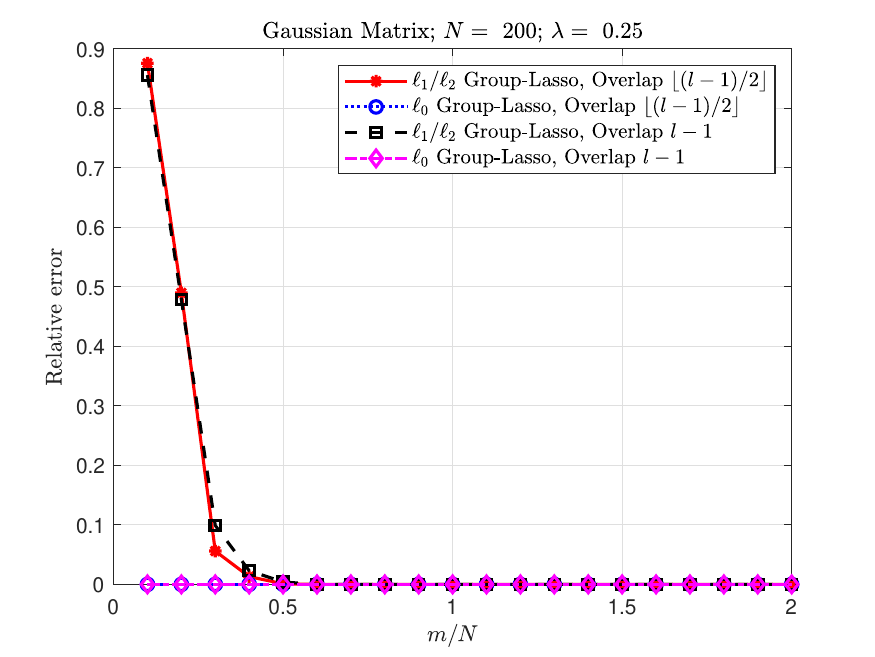}
\includegraphics[scale=0.5]{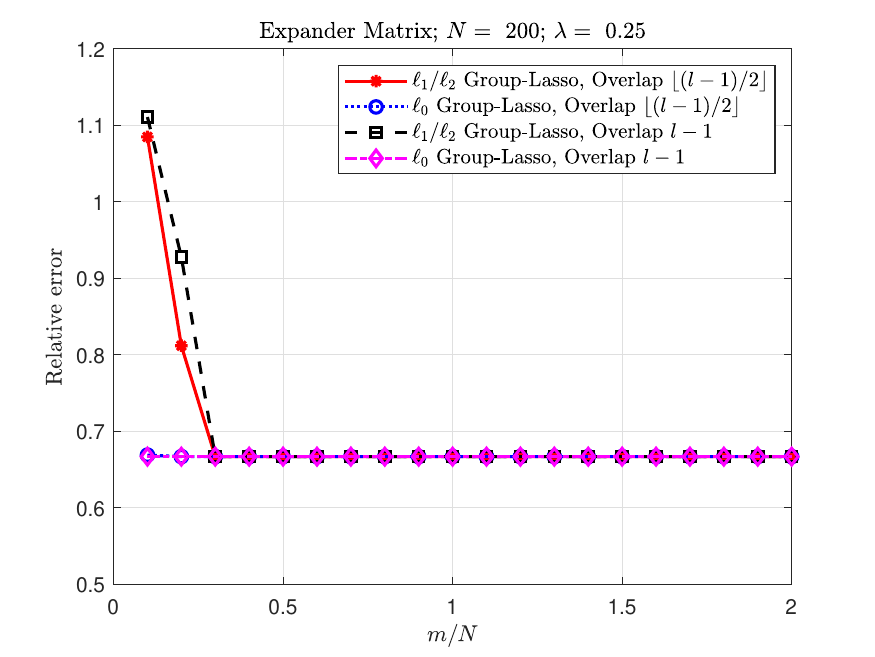}
\caption{Median relative error over all $20$ instances.}
\label{fig:LassoMRE}
\end{figure}

\begin{figure}[h!]
\centering
\includegraphics[scale=0.5]{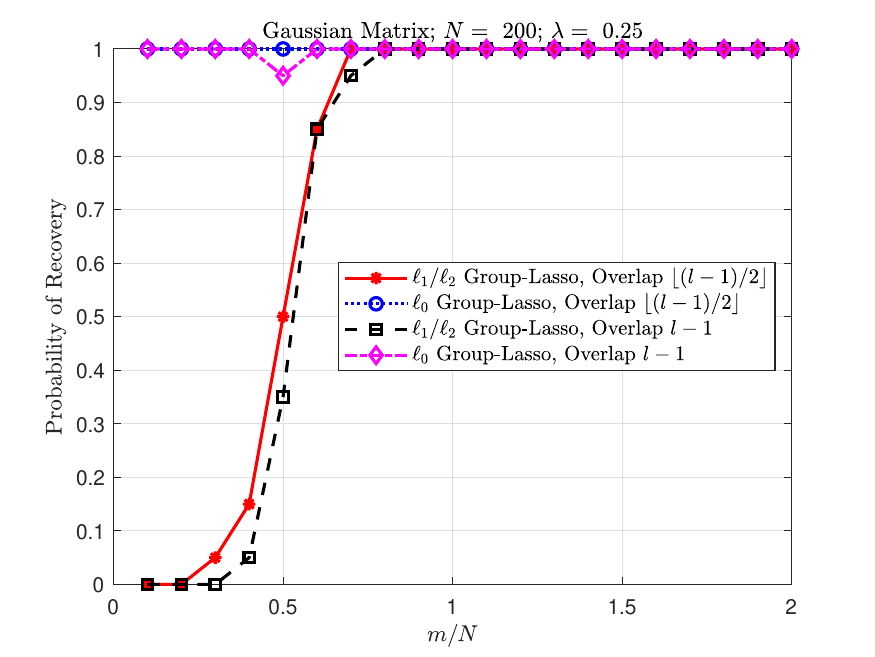}
\includegraphics[scale=0.5]{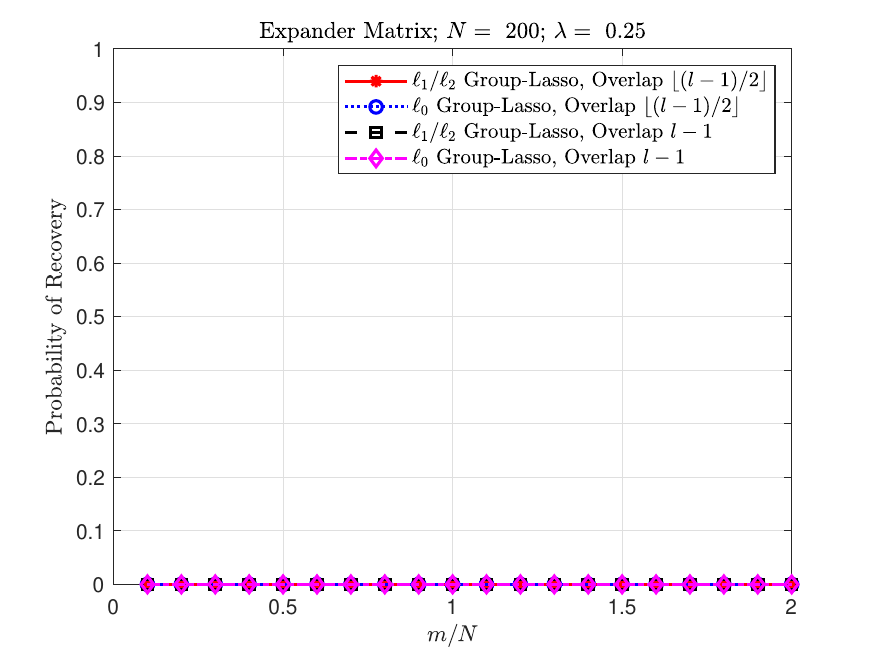}
\caption{Probability of recovery.}
\label{fig:LassoProb}
\end{figure}

\begin{figure}[h!]
\centering
\includegraphics[scale=0.5]{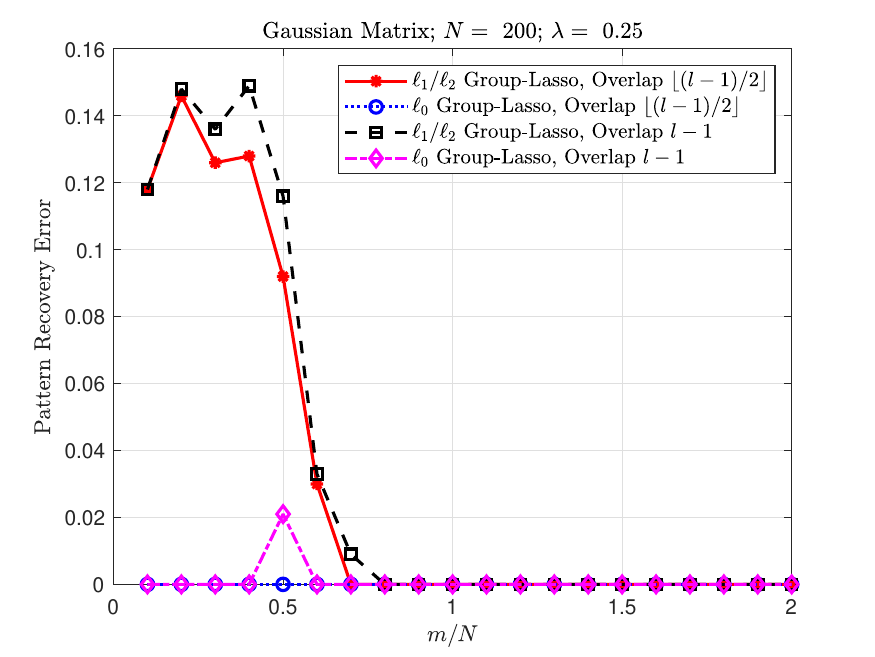}
\includegraphics[scale=0.5]{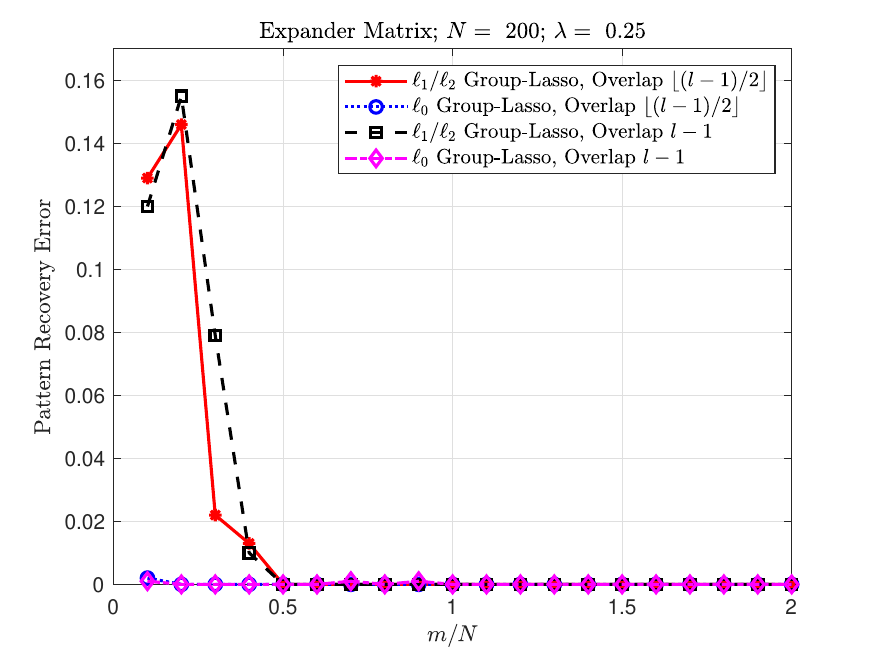}
\caption{Average pattern recovery error over all $20$ instances.}
\label{fig:LassoPRE}
\end{figure}

\begin{figure}[h!]
\centering
\includegraphics[scale=0.5]{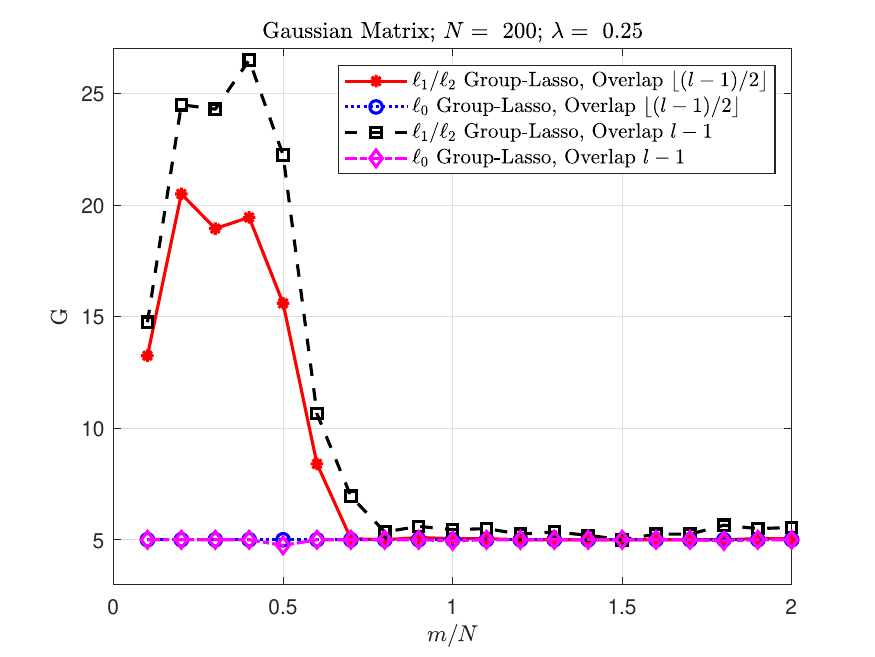}
\includegraphics[scale=0.5]{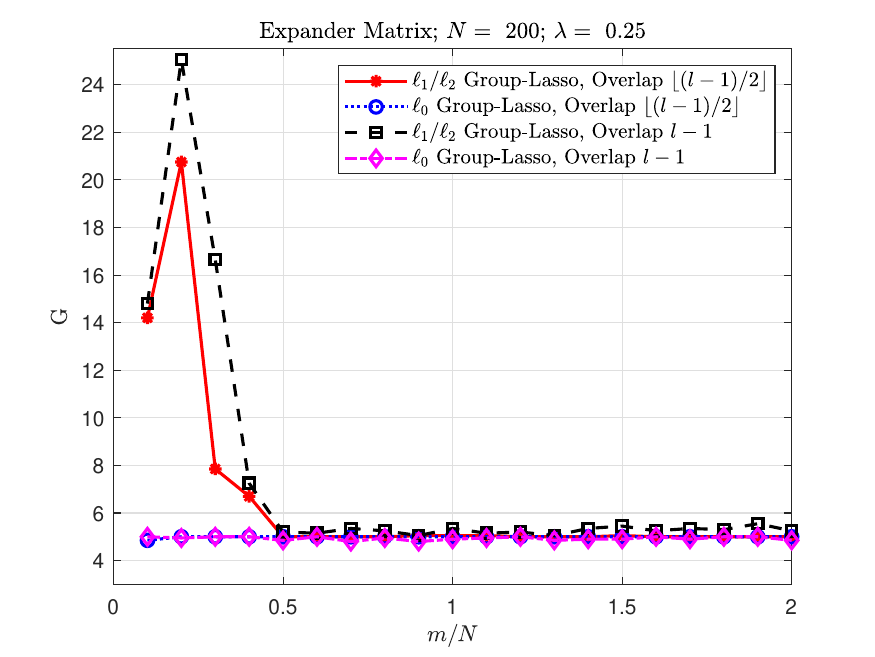}
\caption{Average number of groups $G$ selected in $\hat x$ over all $20$ instances.}
\label{fig:LassoG}
\end{figure}

\begin{figure}[h!]
\centering
\includegraphics[scale=0.5]{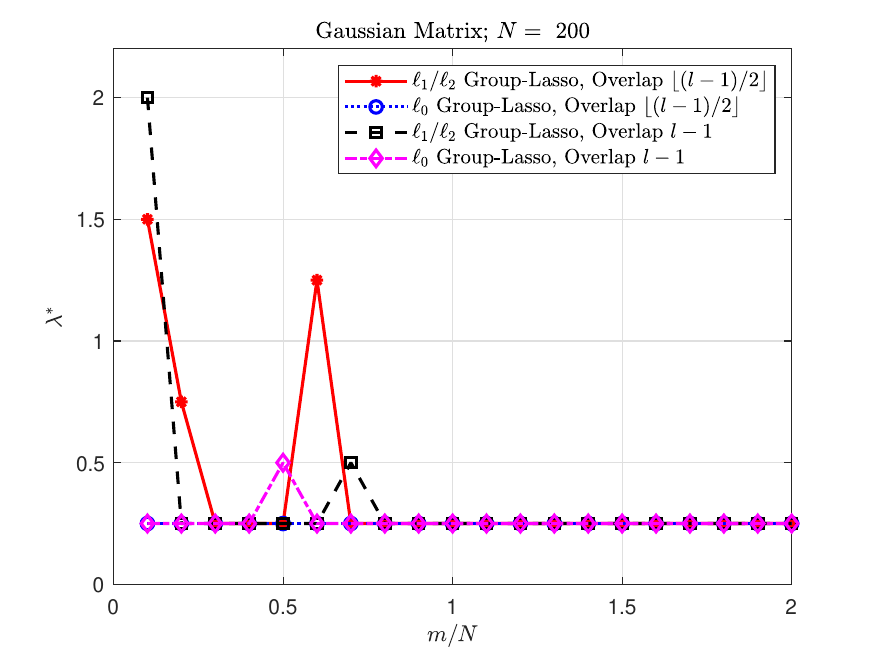}
\includegraphics[scale=0.5]{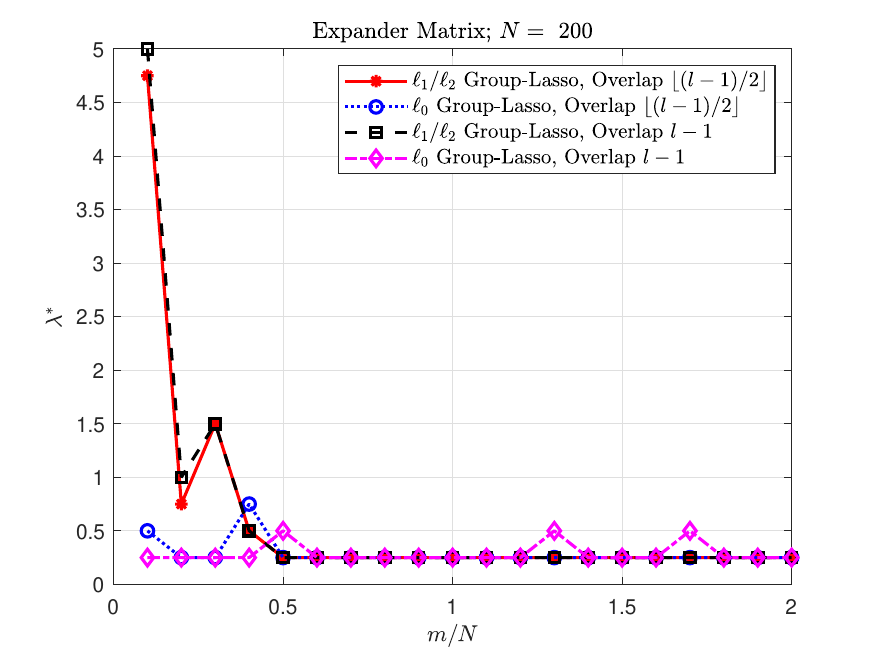}
\caption{Value of $\lambda$ with best mean relative error.}
\label{fig:LassoLambda}
\end{figure}

The results in Figure \ref{fig:LassoMRE} show that the $\ell_0$ variant of the latent group Lasso performs very well for Gaussian matrices. Even for a small number of measurements the median relative error is $0$ while for expander matrices the error never is smaller than $0.6$. The $\ell_1 / \ell_2$ variant for Gaussian matrices has a larger error for small number of measurements which decreases rapidly and is always $0$ for $m$ larger than $0.5N$. In the expander case it is never smaller than $0.6$ as well. The same picture holds for the pattern recovery error; see Figure \ref{fig:LassoPRE}. The only difference here is that also for expander matrices the error tends to $0$. The results indicate that the optimal support is calculated for both variants and for both types of matrices if $m$ is large enough, but for expander matrices the latent group Lasso struggles to find the optimal component-values of $\x$ in the support. Interestingly the frequency of the groups does not significantly influence the results.

The probability of recovery for Gaussian matrices is $1$ for all $m$ for the $\ell_0$ variant and is $1$ for $m$ larger than $0.8N$ for the $\ell_1 / \ell_2$ variant; see Figure \ref{fig:LassoProb}. According to the results for the relative error the probability of recovery for expander matrices is $0$ for all $m$. The number of groups selected by the latent group Lasso is close to $5$ for all $m$ for the $\ell_0$ variant; see Figure \ref{fig:LassoG}. For the $\ell_1 / \ell_2$ variant with Gaussian matrices it is close to $5$ for all $m$ larger than $0.8N$ while for the expander matrices it is already close to $5$ for all $m$ larger than $0.5N$. The value of the optimal $\lambda$ is large for small $m$ and is always $0.25$ for larger $m$; see Figure \ref{fig:LassoLambda}.

To summarize the results, it seems that the $\ell_0$ latent group Lasso outperforms the $\ell_1 / \ell_2$ variant. It can even compete with the iterative algorithms tested in the previous section; the number of required measurements can be even smaller for the $\ell_0$ latent group Lasso while at the same time the support is correctly recovered. Nevertheless in a real-word application the optimal $\lambda$ is not known and has to be found. Furthermore in contrast to the iterative algorithms studied in this work it can never be guaranteed that the recovered support and especially the number of groups calculated by the latent group Lasso is optimal. The expander variant of the latent group Lasso performs much worse than for the iterative algorithms. Especially this approach fails to recover the true signal for all instances, while the true support can be found.

\section{Conclusion}

In this paper we revisited the model-based compressed sensing problem focusing on overlapping group models with bounded treewidth and low frequency. We derived a polynomial time dynamic programming algorithm to solve the exact projection problem for group models with bounded treewidth, which is more general than the state-of-the-art considering loopless overlapping models. For general group models we derived an algorithm based on the idea of Bender's decomposition, which may run in exponential time but often performs better than dynamic programming in practice. We proved that the latter procedure is generalizable from group-sparse models to group-sparse plus standard sparse models. The most dominant operation of iterative exact projection algorithms is the model projection. Hence our results show that the Model-IHT and the MEIHT run in polynomial time for group models with bounded treewidth. 
Alternatively, for group models with bounded frequency we show that another class of less accurate algorithms run in polynomial time. More precisely the AM-IHT and the AM-EIHT are algorithms using head- and tail-approximations instead of exact projections. 

Using Benders' Decomposition (with Gaussian and model-expander sensing matrices) we compare the minimum number of measurements required by, and runtimes of, each of the four algorithms (Model-IHT, MEIHT, AM-IHT and AM-EIHT) to achieve a given accuracy. In summary the experimental results on overlapping block groups seem to indicate that the number of required measurements to recover a signal is smaller for expander matrices than for Gaussian matrices. Furthermore, we could observe that the number of measurements to ensure a small relative error is smaller for the approximate versions of the algorithms. The run-time gets much larger for Gaussian matrices with increasing $N$ than for expander matrices, which might be just what is expected when applying dense versus sparse matrices. In general the approximate versions of the algorithms may have a larger number of iterations but the run-time is lower. This indicates that the larger number of iterations can be compensated by the faster computation of the approximate projection problems in each iteration. Additionally to the iterative algorithms we test the latent group Lasso approach on the same instances and show that the $\ell_0$ variant outperforms the $\ell_1 / \ell_2$ variant and is even competitive to the iterative algorithms.

\paragraph*{Acknowledgements}
BB acknowledges the support from the funding by the German Federal Ministry of Education and Research, administered by Alexander von Humboldt Foundation, for the German Research Chair at AIMS South Africa.

\bibliographystyle{abbrv}      
\bibliography{./sparse}


\end{document}